\documentclass[11pt]{amsart}
\usepackage[T1]{fontenc}
\usepackage{lmodern}

\usepackage{amsmath}
\DeclareMathOperator{\diag}{diag}
\usepackage{amssymb}
\usepackage{amsthm}
\usepackage{mathtools}
\usepackage{mathrsfs}

\usepackage{graphicx}
\usepackage{booktabs}
\usepackage{enumitem}
\usepackage{float}
\usepackage{subcaption}
\usepackage{tikz}
\usetikzlibrary{calc}
\usepackage{tabularx}
\usepackage{array} 
\newcolumntype{L}[1]{>{\raggedright\arraybackslash}p{#1}}
\usepackage{longtable}
\usepackage{booktabs}  
\makeatletter
\renewcommand\subsubsection{\@startsection{subsubsection}{3}{\z@}%
                                     {-3.25ex\@plus -1ex \@minus -.2ex}%
                                     {1.5ex \@plus .2ex}%
                                     {\normalfont\normalsize\bfseries}}
\makeatother

\usepackage[margin=1in]{geometry}

\usepackage{xcolor}
\definecolor{mycyanblue}{HTML}{0080AC}
\usepackage{hyperref}
\hypersetup{
    colorlinks=true,
    allcolors=mycyanblue
}

\usepackage[nameinlink,capitalise]{cleveref}

\newtheorem{theorem}{Theorem}[section]
\newtheorem{lemma}[theorem]{Lemma}
\newtheorem{proposition}[theorem]{Proposition}
\newtheorem{corollary}[theorem]{Corollary}

\theoremstyle{definition}

\newtheorem{problem}[theorem]{Problem}

\theoremstyle{remark}
\newtheorem{remark}[theorem]{Remark}

\newcommand{\one}{\mathbf{1}}

\DeclareMathOperator{\Spec}{Spec}

\title{A Sharp Matching-Number Threshold for Spectral-Walk Determination of Trees}
\author{Chaochao Zhu}
\thanks{College of Finance and Mathematics, West Anhui University, Lu'an, 237012, China. Email: zccjsbz@amss.ac.cn}

\begin{document}

\maketitle

\begin{abstract}
The spectral characterization of graphs is a central problem in spectral graph theory. In this paper we study when a tree is determined, among trees, by its generalized spectrum. We use the equivalent formulation given by the adjacency spectrum together with the total-walk sequence $W_k(G)=\mathbf 1^{\mathsf T}A(G)^k\mathbf 1$. We determine the exact matching-number threshold for this tree-level reconstruction problem. If \(T\) and \(T'\) are trees with matching number at most 4 and have the same adjacency spectrum and the same total-walk sequence, then \(T\cong T'\). Moreover, in this range it is enough to require equality of \(W_k\) for \(3\le k\le8\). The bound is sharp: for every positive integer \(m\) we construct a pair of non-isomorphic trees with matching number 5 having the same adjacency spectrum and identical total-walk sequences. The proof of the positive result is based on a finite-core reduction and an algebraic reconstruction of the possible pendant attachments.
\end{abstract}


% =========================================================
\section{Introduction}\label{sec:introduction}
% =========================================================

Determining the structure of a graph from its spectral invariants is a fundamental and challenging problem in spectral graph theory; see, for example, the surveys of van Dam and Haemers \cite{vanDamHaemers2003,vanDamHaemers2009} and the standard references \cite{Cvetkovic,GodsilRoyle,BrouwerHaemers2011}. A graph $G$ is said to be \emph{determined by its spectrum} (DS) if every graph with the same adjacency spectrum as $G$ is necessarily isomorphic to $G$. Non-isomorphic cospectral graphs are abundant, and a number of
systematic constructions of cospectral graphs are known \cite{GodsilMcKay1982}. In particular, Schwenk showed that almost all trees have a non-isomorphic cospectral mate \cite{Schwenk1973}. This line of research, rooted in chemistry and famously analogized to the question ``Can one hear the shape of a drum?'' \cite{Kac}, has spurred extensive research over the past decades. A related problem is determination by the generalized spectrum (DGS), where one considers the adjacency spectra of both a graph and its
complement.  An influential arithmetic approach to DGS was developed by Wang and Xu \cite{WangXu2006} through the walk matrix, and was subsequently refined in several directions \cite{Wang2017,QiuWangWangZhang2023,WangWangZhu2023}. The generalized spectral characterization of trees has also received specific attention; see, for instance, the recent work of Yang and Wang \cite{YangWang2024}. It is clear that every DS graph is also DGS, but the converse is not true, leading to the natural question of characterizing graphs that are DGS but not DS \cite{Lin2026}. 

In this paper, we study a rigidity problem for trees that is in fact equivalent to determination by the generalized spectrum, but formulated in a more convenient way. For a graph $G$ with adjacency matrix \(A(G)\), the total number of walks of length $k$ is given by $W_k(G) = \one^{\mathsf T} A(G)^k \one$, where $\one$ is the all-one vector. The sequence $(W_0(G), W_1(G), W_2(G), \ldots)$ is known as the total-walk sequence, and it is encoded by the generating function $F_G(x) = \sum_{k\ge 0} W_k(G) x^k = \one^{\mathsf T}(I - xA(G))^{-1}\one$.

It is well known (and will be used implicitly) that for any graph, the pair consisting of the adjacency spectrum $\Spec(G)$ and the total-walk sequence $(W_k(G))_{k\ge 0}$ is equivalent to the generalized spectrum $(\Spec(G), \Spec(\overline{G}))$. Indeed, the spectral decomposition of $A(G)$ determines the projections of $\one$ onto each eigenspace, which in turn determine the characteristic polynomial of the complement. Thus the invariant we consider is precisely the generalized spectrum, but we choose to work with this equivalent representation because it lends itself to algebraic manipulations with the generating function $F_G(x)$, which is particularly tractable for trees.

Having established the equivalence between the generalized spectrum and the total-walk sequence, we are now in a position to state the precise problem investigated in this paper.

\begin{problem}
If two trees are cospectral and have the same total-walk sequence, must they be isomorphic?
\end{problem}

The significance of the problem is twofold. First, the spectral--walk data considered here are equivalent to the generalized spectrum, so the question is a generalized spectral reconstruction problem for trees. Second, the matching number provides a natural measure of structural complexity for a tree.  Our results show that, with respect
to this parameter, generalized spectral rigidity exhibits a sharp threshold: it holds for all trees with matching number at most 4, but already fails at matching number 5. This threshold is not merely a limitation of our method. At matching number 5 we construct an infinite family of non-isomorphic pairs with identical spectral--walk data.  Thus the transition from \(\nu=4\) to \(\nu=5\) marks a genuine change in the reconstruction behaviour of trees.

\begin{theorem}\label{thm:main-rigidity}
Let \(T\) and \(T'\) be trees with $\nu(T)\leq4$, $\nu(T')\leq4$. Suppose that $\operatorname{Spec}(T)=\operatorname{Spec}(T')$, and $\mathbf1^{\mathsf T}A(T)^k\mathbf1=\mathbf1^{\mathsf T}A(T')^k\mathbf1$, $(k\geq0)$. Then $T\cong T'$.
\end{theorem}

Although the problem is formulated in terms of the entire total-walk sequence, the positive result is in fact finite. 

\begin{corollary}\label{cor:finite-main}
In Theorem~\ref{thm:main-rigidity}, the condition $W_k(T)=W_k(T')$, $(k\geq0)$ may be replaced by $W_k(T)=W_k(T')$, $(3\leq k\leq8)$.
\end{corollary}

\begin{proof}
For cospectral trees, \(W_1\) and \(W_2\) are already determined by the adjacency spectrum. Indeed, \(W_1(T)=2(n-1), W_2(T)=n(n-1)-2m_2(T),\) and both \(n\) and \(m_2(T)\) are spectral invariants of a tree. Moreover, every forest \(F\) satisfies \(\operatorname{rank}A(F)=2\nu(F).\) Hence, when \(\nu(T)\le4\), the common adjacency rank is at most 8. The Cayley--Hamilton recurrence on the nonzero spectral subspace then shows that equality of \(W_1,\ldots,W_8\) implies equality of the entire total-walk sequence.
\end{proof}

Since cospectral trees have the same matching number, as recalled in the preliminaries below, either of the conditions \(\nu(T)\le 4\) and \(\nu(T')\le 4\) implies the other under the cospectrality hypothesis. Thus only one matching-number assumption is logically necessary in Theorem~\ref{thm:main-rigidity}; we retain the symmetric formulation for clarity.

The following theorem demonstrates that the bound $\nu\le 4$ in Theorem~\ref{thm:main-rigidity} is best possible.

\begin{theorem}\label{thm:sharpness}
For every integer $m\ge 1$, there exist two non-isomorphic trees $T_m$ and $T_m'$ such that $\nu(T_m)=\nu(T_m')=5$, and they share the same spectrum and the same total-walk sequence (hence the same generalized spectrum).
\end{theorem}

The sharpness family originates from a computational search that identified the first collision for $n=17$ and matching number 5. Both trees share the same core path $P_7$ with different pendant attachments. The original pair is given by $P_7(2,0,0,1,6,0,1)$ and $P_7(1,0,0,1,0,6,2)$. For clarity, we depict the two trees of this initial collision in Figure~\ref{fig:sharpness_pair} below.

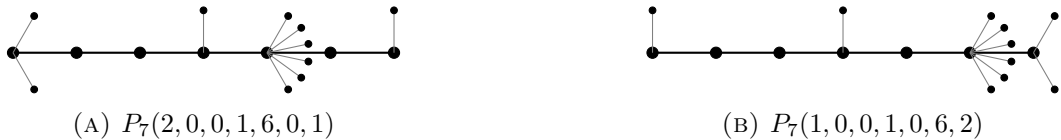
\begin{figure}[H]
    \centering
    \begin{subfigure}[b]{0.48\textwidth}
        \centering
        \begin{tikzpicture}[scale=0.7, every node/.style={inner sep=0pt}]
         
            \coordinate (p1) at (0,0);
            \coordinate (p2) at (1.2,0);
            \coordinate (p3) at (2.4,0);
            \coordinate (p4) at (3.6,0);
            \coordinate (p5) at (4.8,0);
            \coordinate (p6) at (6.0,0);
            \coordinate (p7) at (7.2,0);
            
            \draw[thick] (p1)--(p2)--(p3)--(p4)--(p5)--(p6)--(p7);
            \foreach \p in {1,...,7} {
                \filldraw[black] (p\p) circle (3pt);
            }
            
            \newcommand{\drawLeaves}[2]{
                \ifnum#2=1
                   
                    \draw[gray, thin] (#1) -- ++(90:0.8);
                    \filldraw[black] ($(#1)+(90:0.8)$) circle (1.8pt);
                \else
                    
                    \pgfmathsetmacro{\angleStart}{-60}
                    \pgfmathsetmacro{\angleEnd}{60}
                    \pgfmathsetmacro{\step}{(\angleEnd-\angleStart)/(#2-1)}
                    \foreach \i in {1,...,#2} {
                        \pgfmathsetmacro{\angle}{\angleStart + (\i-1)*\step}
                        \draw[gray, thin] (#1) -- ++(\angle:0.8);
                        \filldraw[black] ($(#1)+(\angle:0.8)$) circle (1.8pt);
                    }
                \fi
            }
                      
            \drawLeaves{p1}{2}
            \drawLeaves{p4}{1}
            \drawLeaves{p5}{6}
            \drawLeaves{p7}{1}
        \end{tikzpicture}
        \caption{$P_7(2,0,0,1,6,0,1)$}
        \label{fig:tree1}
    \end{subfigure}
    \hfill
    \begin{subfigure}[b]{0.48\textwidth}
        \centering
        \begin{tikzpicture}[scale=0.7, every node/.style={inner sep=0pt}]
            \coordinate (p1) at (0,0);
            \coordinate (p2) at (1.2,0);
            \coordinate (p3) at (2.4,0);
            \coordinate (p4) at (3.6,0);
            \coordinate (p5) at (4.8,0);
            \coordinate (p6) at (6.0,0);
            \coordinate (p7) at (7.2,0);
            
            \draw[thick] (p1)--(p2)--(p3)--(p4)--(p5)--(p6)--(p7);
            \foreach \p in {1,...,7} {
                \filldraw[black] (p\p) circle (3pt);
            }
            
            \newcommand{\drawLeaves}[2]{
                \ifnum#2=1
                    \draw[gray, thin] (#1) -- ++(90:0.8);
                    \filldraw[black] ($(#1)+(90:0.8)$) circle (1.8pt);
                \else
                    \pgfmathsetmacro{\angleStart}{-60}
                    \pgfmathsetmacro{\angleEnd}{60}
                    \pgfmathsetmacro{\step}{(\angleEnd-\angleStart)/(#2-1)}
                    \foreach \i in {1,...,#2} {
                        \pgfmathsetmacro{\angle}{\angleStart + (\i-1)*\step}
                        \draw[gray, thin] (#1) -- ++(\angle:0.8);
                        \filldraw[black] ($(#1)+(\angle:0.8)$) circle (1.8pt);
                    }
                \fi
            }
            
            \drawLeaves{p1}{1}
            \drawLeaves{p4}{1}
            \drawLeaves{p6}{6}
            \drawLeaves{p7}{2}
        \end{tikzpicture}
        \caption{$P_7(1,0,0,1,0,6,2)$}
        \label{fig:tree2}
    \end{subfigure}
    \caption{The two non-isomorphic trees forming the initial collision.}
    \label{fig:sharpness_pair}
\end{figure}
By replacing the parameter $6$ with an arbitrary integer $m$, we obtain an infinite family.

The proof of the positive result is based on a finite-core reduction. Removing all pendant vertices from a tree of bounded matching number leaves a core whose order is bounded in terms of the matching number. For matching number 4, this reduces the problem to finitely many canonical cores and pendant-support types. The adjacency spectrum
determines the matching data of the tree, while the total-walk generating function provides additional polynomial invariants that recover the pendant multiplicities and separate the remaining canonical families. The resulting proofs are entirely algebraic; computation served only to motivate the problem and to discover the sharpness family, and is not used in the proofs of the main results.

The paper is organized as follows. Section~\ref{sec:preliminaries} contains preliminary definitions and basic facts about total-walk sequences and matching polynomials of trees. In Section~\ref{sec:core-reduction}, we develop a core reduction technique that reduces the study of trees with bounded matching number to a finite analysis of their cores. Section~\ref{sec:matching-at-most-four} presents the proof of the rigidity theorem for trees with matching number at most 4. Section~\ref{sec:sharpness-section} provides the explicit construction for the infinite sharpness family with matching number 5. Finally, Appendix~A collects the explicit recovery certificates
underlying Proposition~\ref{prop:internal-reconstruction}.

% =========================================================
\section{Preliminaries}\label{sec:preliminaries}
% =========================================================

Throughout the paper, all graphs are finite, simple, and undirected. For a graph \(G\), let \(V(G)\) and \(E(G)\) denote its vertex and edge sets, respectively, and let \(A(G)\) denote its adjacency matrix. We write \(\operatorname{Spec}(G)\) for the adjacency spectrum of \(G\). For a graph \(G\), let \(m_j(G)\) denote the number of matchings of size \(j\), and let \(\nu(G)\) denote its matching number. For background on matching polynomials and their connections with walks, we refer to Godsil \cite{Godsil1981}. A vertex cover of \(G\) is a set meeting every edge of \(G\); its minimum cardinality is denoted by \(\tau(G)\).

We shall use the following classical results.

\begin{lemma}[Kőnig–Egerváry{\cite{Konig1931,Egervary1931,LovaszPlummer1986}}]\label{lem:konig}
For every bipartite graph \(G\), $\nu(G)=\tau(G)$.
\end{lemma}

For a graph \(G\) on \(n\) vertices, write $\phi(G,\lambda): =\det\bigl(\lambda I_n-A(G)\bigr)$ for its adjacency characteristic polynomial.

\begin{lemma}[Sachs' theorem{\cite{BrouwerHaemers2011}}]\label{lem:sachs-forest}
Let \(T\) be a forest on \(n\) vertices. Then
    \[\phi(T,\lambda)
      =
      \sum_{j=0}^{\nu(T)}
      (-1)^j m_j(T)\lambda^{\,n-2j}.\]
\end{lemma}

It is convenient to use the reciprocal form of the characteristic polynomial. For a graph \(G\) on \(n\) vertices, define $\Delta_G(x):=\det\bigl(I_n-xA(G)\bigr)$. For a forest, Lemma~\ref{lem:sachs-forest} gives \[\Delta_T(x)=\sum_{j=0}^{\nu(T)}(-1)^j m_j(T)x^{2j}.\]

We record the following immediate consequence of Lemma~\ref{lem:sachs-forest} for later use.

\begin{corollary}\label{cor:rank-matching}
Let \(T\) be a forest on \(n\) vertices. Then \(\deg\Delta_T=2\nu(T).\) Moreover, the multiplicity of \(0\) as an adjacency eigenvalue of \(T\) is \(n-2\nu(T)\); in particular, \(\operatorname{rank}A(T)=2\nu(T).\)
\end{corollary}

\begin{proof}
Since \(m_{\nu(T)}(T)>0\), the formula for \(\Delta_T(x)\) shows that \(\deg\Delta_T=2\nu(T).\) Likewise, the lowest nonzero power of \(\lambda\) in
\(\phi(T,\lambda)\) is \(\lambda^{n-2\nu(T)}\). Hence \(0\) has multiplicity \(n-2\nu(T)\), and therefore \(\operatorname{rank}A(T)=2\nu(T).\)
\end{proof}

\begin{corollary}\label{cor:cospectral-matchings}
Cospectral forests have the same number of matchings of every size. In particular, cospectral forests have the same matching number.
\end{corollary}

\begin{proof}
Cospectral forests have the same characteristic polynomial. Comparing the coefficients in displayed formula for \( \phi(T,\lambda) \) gives $m_j(T)=m_j(T')$, for every $j\geq 0$. The assertion about the matching number follows immediately.
\end{proof}

We now define the spectral–walk data. Let \(G\) be a graph on \(n\) vertices, and let \(\mathbf{1}_n\) denote the all-one vector of length \(n\). For \(k\geq 0\), define $W_k(G):=\mathbf{1}_n^{\mathsf T}A(G)^k\mathbf{1}_n$. Thus \(W_k(G)\) is the total number of walks of length \(k\) in \(G\). The sequence $\bigl(W_0(G),W_1(G),W_2(G),\ldots\bigr)$ is called the \emph{total-walk sequence} of \(G\). Its generating function is $F_G(x):=\sum_{k\geq 0}W_k(G)x^k$. Since $\bigl(I_n-xA(G)\bigr)^{-1}=\sum_{k\geq 0}x^kA(G)^k$ as a formal power series, we have $F_G(x)=\mathbf{1}_n^{\mathsf T}\bigl(I_n-xA(G)\bigr)^{-1}\mathbf{1}_n$. Two graphs \(G\) and \(H\) are said to have the same \emph{spectral-walk data} if $\operatorname{Spec}(G)=\operatorname{Spec}(H)$ and $W_k(G)=W_k(H)$ for every $k\geq 0$.

\begin{remark}\label{rem:generalized-spectrum}
The spectral-walk data are equivalent to the generalized spectrum $(\operatorname{Spec}(G), \operatorname{Spec}(\overline{G}))$.
\end{remark}
This viewpoint is closely related to the walk-matrix approach to generalized spectral characterization \cite{WangXu2006,Wang2017,QiuWangWangZhang2023}. We use the spectral-walk formulation because the generating function \(F_G(x)\) is particularly convenient for trees with pendant attachments.

For the algebraic arguments below, it is convenient to encode the same data by a pair of polynomials. For a graph \(G\) on \(n\) vertices, define $N_G(x):=\mathbf{1}_n^{\mathsf T}\operatorname{adj}\bigl(I_n-xA(G)\bigr)\mathbf{1}_n$, where \(\operatorname{adj}(M)\) denotes the adjugate of \(M\). The identity $M^{-1}=\operatorname{adj}(M)/\det M$ gives $F_G(x) =N_G(x)/\Delta_G(x)$. Thus the spectrum and the total-walk sequence can be encoded by the two polynomials \(\Delta_G\) and \(N_G\).

\begin{proposition}\label{prop:spectral-walk-polynomials}
For any two graphs \( G \) and \( H \), the following are equivalent:
\begin{enumerate}[label=(\roman*)]
    \item \( G \) and \( H \) have the same spectral-walk data;
    \item \( \Delta_G(x) = \Delta_H(x) \) and \( N_G(x) = N_H(x) \).
\end{enumerate}
\end{proposition}

\begin{proof}
Suppose first that \(G\) and \(H\) have the same spectral-walk data. Cospectrality gives $\Delta_G(x)=\Delta_H(x)$, while equality of the total-walk sequences gives
$F_G(x)=F_H(x)$. It follows from $F_G(x) =N_G(x)/\Delta_G(x)$ that $N_G(x)=\Delta_G(x)F_G(x)=\Delta_H(x)F_H(x)=N_H(x)$.

Conversely, suppose that $\Delta_G(x)=\Delta_H(x)$ and $N_G(x)=N_H(x)$. Then $F_G(x)=N_G(x)/\Delta_G(x)$ gives $F_G(x)=N_G(x)/\Delta_G(x)=N_H(x)/\Delta_H(x)=F_H(x)$. Therefore $W_k(G)=W_k(H)$ for every $k\geq 0$. It remains to recover cospectrality. Since $N_G(0)=\mathbf{1}_{|V(G)|}^{\mathsf T} \mathbf{1}_{|V(G)|}=|V(G)|$, we obtain
$|V(G)|=N_G(0)=N_H(0)=|V(H)|$.

The common polynomial \(\Delta_G=\Delta_H\) determines all nonzero adjacency eigenvalues, with their multiplicities.  Because \(G\) and \(H\) have the same order, the multiplicity of the eigenvalue \(0\) is also the same. Hence $\operatorname{Spec}(G)=\operatorname{Spec}(H)$. Thus \(G\) and \(H\) have the same spectral-walk data.
\end{proof}

The following fixed-core reduction will be essential later.

\begin{lemma}[Schur complement {\cite{BrouwerHaemers2011}}]\label{lem:schur-complement}
Let $ M=
    \begin{pmatrix}
        B&C\\
        D&E
    \end{pmatrix}$,
where \(E\) is invertible. Then $\det M=\det(E)\det\bigl(B-CE^{-1}D\bigr)$. If \(M\) is invertible, then the upper-left block of \(M^{-1}\) is $\bigl(B-CE^{-1}D\bigr)^{-1}$.
\end{lemma}

In Section~\ref{sec:core-reduction}, we apply Lemma~\ref{lem:schur-complement} to the block decomposition of \(I-xA(T)\) obtained by separating the core vertices from the pendant vertices. The resulting formulas for \(\Delta_T\) and \(F_T\) will be proved there explicitly.

% =========================================================
\section{Core reduction}\label{sec:core-reduction}
% =========================================================

We now develop the finite-core reduction used in the proof of Theorem~\ref{thm:main-rigidity}. In this section we represent a tree by an intrinsic core together with the numbers of pendant vertices attached to the core vertices. For fixed matching number, this reduces the problem to finitely many parametric families.

For a tree \(T\), let \(L(T)\) denote the set of pendant vertices, and $N_T(v):=\{u\in V(G):uv\in E(G)\}$ the open neighbourhood of \(v\) respectively. Whenever \(|V(T)|\geq 3\), define the \emph{core} of \(T\) by $R(T):=T-L(T)$. The graph \(R(T)\) is connected: the unique path in \(T\) between two non-pendant vertices cannot contain a pendant vertex as an internal vertex. Thus \(R(T)\) is itself a tree, possibly consisting of a single vertex.

We first bound the order of the core in terms of the matching number.

\begin{lemma}\label{lem:bounded-core}
Let \(T\) be a tree with \(|V(T)|\geq 3\) and \(\nu(T)=t\). Then $|V(R(T))|\leq 2t-1$.
\end{lemma}

\begin{proof}
By Lemma~\ref{lem:konig}, \(T\) has a minimum vertex cover \(C\) of cardinality $|C|=\tau(T)=\nu(T)=t$. We may choose \(C\) to contain no pendant vertex.  Indeed, suppose that a pendant vertex \(v\in C\) has unique neighbour \(u\). If \(u\in C\), then \(C\setminus\{v\}\) is still a vertex cover, contradicting the minimum cardinality of \(C\).  Hence \(u\notin C\). Since \(|V(T)|\geq 3\), the vertex \(u\) is not pendant, and $(C\setminus\{v\})\cup\{u\}$ is a minimum vertex cover containing fewer pendant vertices.  Repeating this replacement gives the required choice of \(C\).

Put \(I:=V(T)\setminus C\). Since \(C\) is a vertex cover, \(I\) is independent. Let \(r\) be the number of non-pendant vertices in \(I\). Because \(C\) contains no pendant vertex, the core has precisely \(t+r\) vertices: $V(R(T))=C\cup\{v\in I:\deg_T(v)\geq 2\}$. Every one of the \(r\) core vertices belonging to \(I\) has at least two neighbours in \(C\). All these incident edges lie in \(R(T)\). Since \(R(T)\) is a tree on \(t+r\) vertices, $2r\leq |E(R(T))|=t+r-1$. Thus \(r\leq t-1\), and consequently $|V(R(T))|=t+r\leq 2t-1$.
\end{proof}

Let \(R\) be a tree with $V(R)=\{v_1,\ldots,v_q\}$. For $\mathbf{a}=(a_1,\ldots,a_q)^{\mathsf T}\in\mathbb Z_{\geq 0}^{q}$, let \(R(\mathbf{a})\) denote the tree obtained from \(R\) by attaching \(a_i\) new pendant vertices to \(v_i\), for each \(1\leq i\leq q\).  We write $|\mathbf{a}|:=a_1+\cdots+a_q$ and $\operatorname{supp}(\mathbf{a}):=\{v_i\in V(R):a_i>0\}$. If \(R=R(T)\), then \(a_i\) is uniquely determined by $a_i=|N_T(v_i)\cap L(T)|$, and $T=R(T)(\mathbf{a})$. Thus the pair consisting of the core and the pendant multiplicity function is intrinsic to \(T\), up to an automorphism of the core. The canonical parameters satisfy $\deg_R(v_i)+a_i\geq 2, (1\leq i\leq q)$. In particular, every pendant vertex of \(R\) belongs to \(\operatorname{supp}(\mathbf{a})\).

\begin{lemma}\label{lem:support-bound}
Let \(T=R(T)(\mathbf{a})\). Then $|\operatorname{supp}(\mathbf{a})|\leq \nu(T)$.
\end{lemma}

\begin{proof}
For each \(v\in\operatorname{supp}(\mathbf{a})\), choose one pendant neighbour \(\ell_v\) of \(v\). The edges $\{v\ell_v:v\in\operatorname{supp}(\mathbf{a})\}$ are pairwise vertex-disjoint and hence form a matching in \(T\).
\end{proof}

The next formula shows that, for a fixed core, the matching number depends only on the support of the pendant vector, and not on the positive values of its entries.

\begin{lemma}\label{lem:matching-support-formula}
Let \(R\) be a tree, let \(\mathbf{a}\in\mathbb Z_{\geq 0}^{V(R)}\), and put $S:=\operatorname{supp}(\mathbf{a})$. Then \[\nu(R(\mathbf{a}))=\max_{X\subseteq S}\bigl(|X|+\nu(R-X)\bigr).\]
\end{lemma}

\begin{proof}
Let \(M\) be a matching in \(R(\mathbf{a})\), and let \(X\subseteq S\) be the set of core vertices incident in \(M\) with a pendant edge. At most one pendant edge of \(M\) is incident with each core vertex. After the \(|X|\) pendant edges are removed, all remaining edges of \(M\) lie in \(R-X\). Hence
\[
|M|\le |X|+\nu(R-X)
\le \max_{Y\subseteq S}\bigl(|Y|+\nu(R-Y)\bigr).
\]
Taking the maximum over all matchings \(M\) in \(R(\mathbf a)\) gives
\[
\nu(R(\mathbf a))
\le
\max_{Y\subseteq S}\bigl(|Y|+\nu(R-Y)\bigr).
\]

Conversely, fix \(X\subseteq S\). For each \(v\in X\), choose one pendant neighbour \(\ell_v\) of \(v\). The edges $\{v\ell_v:v\in X\}$, together with a maximum matching of \(R-X\), form a matching in \(R(\mathbf{a})\) of size $|X|+\nu(R-X)$. Taking the maximum over \(X\subseteq S\) proves the reverse inequality.
\end{proof}

For later use, define $\mu_R(S):=\max_{X\subseteq S}\bigl(|X|+\nu(R-X)\bigr)$. Thus $\nu(R(\mathbf{a}))=\mu_R(\operatorname{supp}(\mathbf{a}))$.

\begin{corollary}\label{cor:finite-family-reduction}
For every fixed integer \(t\geq 1\), trees \(T\) satisfying $|V(T)|\geq 3$ and $\nu(T)=t$ belong to finitely many canonical core-support families $(R,S)$, where $|V(R)|\leq 2t-1, |S|\leq t, L(R)\subseteq S, \mu_R(S)=t$. Within a fixed family, the trees are obtained by assigning arbitrary positive integer pendant multiplicities to the vertices of \(S\), with zero multiplicity outside \(S\), modulo the action of \(\operatorname{Aut}(R)\).
\end{corollary}

\begin{proof}
The bounds on \(|V(R)|\) and \(|S|\) follow from Lemmas~\ref{lem:bounded-core} and~\ref{lem:support-bound}. Moreover, \(L(R)\subseteq S\): if \(v\in L(R)\) and \(v\notin S\),
then \(v\) has no pendant neighbour in \(T\), so \(\deg_T(v)=\deg_R(v)=1\), contradicting \(v\in V(R(T))\). By Lemma~\ref{lem:matching-support-formula}, \(\mu_R(S)=\nu(R(\mathbf a))=\nu(T)=t.\) There are only finitely many trees \(R\), up to isomorphism, of order at most \(2t-1\), and each such tree has only finitely many subsets
of cardinality at most \(t\). Hence only finitely many admissible pairs \((R,S)\) occur, modulo core automorphisms.

Conversely, for any such pair \((R,S)\), assign arbitrary positive integer multiplicities to the vertices of \(S\) and zero multiplicity outside \(S\). Since \(L(R)\subseteq S\), the resulting tree \(R(\mathbf a)\) has canonical core \(R\), and Lemma~\ref{lem:matching-support-formula} gives \(\nu(R(\mathbf a))=\mu_R(S)=t.\) Thus each admissible pair \((R,S)\) determines precisely such a canonical core-support family, with the pendant multiplicities identified modulo the action of \(\operatorname{Aut}(R)\).
\end{proof}

The preceding reduction leaves a fixed core together with its pendant multiplicities. We now express the spectral–walk data of  \(R(\mathbf{a})\) directly in terms of the core \(R\) and the pendant vector \(\mathbf{a}\). Let \(A_R\) be the adjacency matrix of \(R\), and define $D_\mathbf{a}:=\operatorname{diag}(a_1,\ldots,a_q)$ and $M_{R,\mathbf{a}}(x):=I_q-xA_R-x^2D_\mathbf{a}$. Also put $u_\mathbf{a}(x):=\mathbf{1}_q+x\mathbf{a}$.

\begin{proposition}\label{prop:fixed-core-compression}
For every tree \(R\) and every \(\mathbf{a}\in\mathbb Z_{\geq 0}^{V(R)}\), $\Delta_{R(\mathbf{a})}(x)=\det M_{R,\mathbf{a}}(x)$, and $F_{R(\mathbf{a})}(x)=|\mathbf{a}|+u_\mathbf{a}(x)^{\mathsf T}M_{R,\mathbf{a}}(x)^{-1}u_\mathbf{a}(x)$. Consequently, $N_{R(\mathbf{a})}(x)=|\mathbf{a}|\det M_{R,\mathbf{a}}(x)+u_\mathbf{a}(x)^{\mathsf T}
 \operatorname{adj}\bigl(M_{R,\mathbf{a}}(x)\bigr)u_\mathbf{a}(x)$.
\end{proposition}

\begin{proof}
Let \(L:=|\mathbf{a}|\), and order the vertices of \(R(\mathbf{a})\) by listing first the \(q\) core vertices and then the \(L\) pendant vertices. With respect to this ordering, the adjacency matrix has the block form
\[
A(R(a))=
\begin{pmatrix}
A_R & C\\
C^T & 0
\end{pmatrix}.
\]
Here $C$ is the $q\times L$ incidence matrix between the core vertices and the pendant vertices. Since \(v_i\) has exactly \(a_i\) pendant neighbours, $CC^{\mathsf T}=D_\mathbf{a}, C\mathbf{1}_L=\mathbf{a}$. Hence
\[
    I-xA(R(\mathbf{a}))
      =
      \begin{pmatrix}
          I_q-xA_R&-xC\\
          -xC^{\mathsf T}&I_L
      \end{pmatrix}.
\]
Applying Lemma~\ref{lem:schur-complement} to the lower-right block \(I_L\) gives $\Delta_{R(\mathbf{a})}(x)=\det\bigl(I_q-xA_R-x^2CC^{\mathsf T}\bigr)=\det\bigl(I_q-xA_R-x^2D_\mathbf{a}\bigr)= \det M_{R,\mathbf{a}}(x)$, which proves the first determinant formula.

To obtain the total-walk generating function, write
\[
    \begin{pmatrix}
        y\\ s
    \end{pmatrix}
      :=
      \bigl(I-xA(R(\mathbf{a}))\bigr)^{-1}
      \begin{pmatrix}
          \mathbf{1}_q\\
          \mathbf{1}_L
      \end{pmatrix}.
\]
The pendant block equation gives $s=\mathbf{1}_L+xC^{\mathsf T}y$. Substitution into the core block equation yields $\bigl(I_q-xA_R-x^2CC^{\mathsf T}\bigr)y=\mathbf{1}_q+xC\mathbf{1}_L$. Using the identities $CC^{\mathsf T}=D_\mathbf{a}, C\mathbf{1}_L=\mathbf{a}$, we obtain $M_{R,\mathbf{a}}(x)y=u_\mathbf{a}(x)$, and hence $y=M_{R,\mathbf{a}}(x)^{-1}u_\mathbf{a}(x)$. Summing the entries of \(y\) and \(s\), we find $F_{R(\mathbf{a})}(x)=\mathbf{1}_q^{\mathsf T}y+\mathbf{1}_L^{\mathsf T}s=|\mathbf{a}|+ \bigl(\mathbf{1}_q^{\mathsf T}+x\mathbf{a}^{\mathsf T}\bigr)y=|\mathbf{a}|+ u_\mathbf{a}(x)^{\mathsf T} M_{R,\mathbf{a}}(x)^{-1}u_\mathbf{a}(x)$, which proves the formula for \(F_{R(\mathbf{a})}(x)\).
               
Finally, substituting $M_{R,\mathbf{a}}(x)^{-1}=\operatorname{adj}(M_{R,\mathbf{a}}(x))/\det M_{R,\mathbf{a}}(x)$ into expression for \(F_{R(\mathbf{a})}(x)\), and using $F_G(x) =N_G(x)/\Delta_G(x)$ and the formulas for \(\Delta_{R(a)}(x)\), gives the expression for \(N_{R(\mathbf{a})}(x)\).
\end{proof}

For convenience, we shall occasionally write $\Delta_{R,\mathbf{a}}(x):=\det M_{R,\mathbf{a}}(x)$ and
\begin{equation}\label{eq:compressed-N}
    N_{R,\mathbf{a}}(x)
      :=
      |\mathbf{a}|\Delta_{R,\mathbf{a}}(x)
      +
      u_\mathbf{a}(x)^{\mathsf T}
      \operatorname{adj}\bigl(M_{R,\mathbf{a}}(x)\bigr)
      u_\mathbf{a}(x).
\end{equation}
Proposition~\ref{prop:fixed-core-compression} then states that $\bigl(\Delta_{R(\mathbf{a})},N_{R(\mathbf{a})}\bigr) = \bigl(\Delta_{R,\mathbf{a}},N_{R,\mathbf{a}}\bigr)$.
Thus, once the core \(R\) is fixed, the spectral-walk data of \(R(\mathbf{a})\) are encoded by two explicit polynomials in \(x\) whose coefficients are integer polynomials in the pendant parameters.

% =========================================================
\section{Trees with matching number at most 4}\label{sec:matching-at-most-four}
% =========================================================

We dispose of the cases \(\nu(T)\leq3\) before turning to the main case \(\nu(T)=4\). The arguments also illustrate the reconstruction mechanisms used later: recovery from elementary symmetric functions and recovery from explicit data-determined products.

\subsection{Matching number at most 2}

For positive integers \(a,b\), let \(D(a,b)\) be obtained from \(K_2\) by attaching \(a\) and \(b\) pendant vertices to its two endpoints. Let \(P(a,b)\) be obtained from \(P_3\) by attaching \(a\) and \(b\) pendant vertices to its two endpoints.

\begin{proposition}\label{prop:matching-at-most-two}
A tree with matching number 2 belongs, up to interchanging \(a,b\), to exactly one of the families $D(a,b)$, $P(a,b)$, $a,b\in\mathbb Z_{>0}$. Moreover, among trees with matching number at most 2, the polynomial pair \((\Delta_T,N_T)\) determines \(T\) up to isomorphism.
\end{proposition}

\begin{proof}
Let \(T=R(\mathbf{a})\) be its canonical core representation. By Corollary~\ref{cor:finite-family-reduction}, $|V(R)|\leq3$, $|\operatorname{supp}(\mathbf{a})|\leq2$, and $L(R)\subseteq\operatorname{supp}(\mathbf{a})$. A one-vertex core gives a star and hence matching number 1. If \(R=K_2\), both core vertices lie in the support and \(T=D(a,b)\).  If \(R=P_3\), its two leaves lie in the support and the support bound excludes the middle vertex, giving \(T=P(a,b)\). The fixed-core formula gives
$\Delta_{D(a,b)}(x)=1-(a+b+1)x^2+abx^4$, $\Delta_{P(a,b)}(x)=1-(a+b+2)x^2+(ab+a+b)x^4$.
      
Define $\vartheta(T):=[x^3]N_T(x)-[x^2]N_T(x)$. Direct substitution in Proposition~\ref{prop:fixed-core-compression} gives $\vartheta(D(a,b))=0$, $\vartheta(P(a,b))=-2ab<0$.
Thus the two families are distinguished. Within either family, \(\Delta_T\) determines \(a+b\) and \(ab\), so \(\{a,b\}\) is recovered as the root multiset of $z^2-(a+b)z+ab$. Interchanging the roots is induced by the reflection of the core. To complete the proof of the second assertion, suppose that two trees \(T,T'\) with matching number at most \(2\) have the same polynomial pair. Then \(\Delta_T=\Delta_{T'}\), and Corollary~\ref{cor:rank-matching} gives \(\nu(T)=\nu(T').\) The cases \(\nu(T)=0,1\) are immediate: \(T=K_1\) when \(\nu(T)=0\), and \(T\) is a star when \(\nu(T)=1\), whose order determines it.
\end{proof}

\subsection{Matching number 3}
For the path $P_q=v_1v_2\cdots v_q$, write \(P_q(a_1,\ldots,a_q)\) for the tree obtained by attaching \(a_i\) pendant vertices to \(v_i\). We similarly write
\(K_{1,3}(r;a,b,c)\) when \(r\) leaves are attached to the centre of \(K_{1,3}\) and \(a,b,c\) leaves are attached to its three leaves.

\begin{lemma}\label{lem:six-matching-three-families}
Every tree with matching number 3 belongs, up to a core automorphism, to exactly one of the following families:
\[
\begin{array}{c|l}
T_1 & P_3(a,b,c)\\
T_2 & P_4(a,0,0,d)\\
T_3 & P_4(a,b,0,d)\\
T_4 & K_{1,3}(0;a,b,c)\\
T_5 & P_5(a,0,0,0,e)\\
T_6 & P_5(a,0,c,0,e),
\end{array}
\qquad
a,b,c,d,e\in\mathbb Z_{>0}.
\]
\end{lemma}

\begin{proof}
Let \(T=R(\mathbf a)\), put \(q=|V(R)|\), and let \(S=\operatorname{supp}(\mathbf a)\). Since every edge of \(T\) is incident with a vertex of \(R\), every matching in \(T\) uses distinct vertices of \(R\). Hence \(\nu(T)\le q\). By Lemma~\ref{lem:bounded-core}, \(q\le 2\nu(T)-1=5.\) Since \(\nu(T)=3\le q\), we obtain \(3\le q\le5.\) Moreover, \(|S|\le3\) and \(L(R)\subseteq S\).

For \(q=3\), the core is \(P_3=v_1v_2v_3\). Since \(L(R)\subseteq S\), both endvertices \(v_1,v_3\) belong to \(S\). If \(v_2\notin S\), then \(S=\{v_1,v_3\}\), and the matching formula $\nu(R(\mathbf a))=\max_{X\subseteq S}\bigl(|X|+\nu(R-X)\bigr)$ gives \(\nu(T)=2\), a contradiction. Hence \(S=V(R)\), giving \(T_1=P_3(a,b,c)\) with \(a,b,c>0\).

For \(q=4\), the core is \(P_4\) or \(K_{1,3}\). On \(P_4\), the two endvertices belong to \(S\), and at most one internal vertex may also belong to \(S\), giving \(T_2,T_3\) up to reflection. On \(K_{1,3}\), the three leaves belong to \(S\). Since \(|S|\le3\), the center does not belong to \(S\), giving \(T_4\).

For \(q=5\), the core \(K_{1,4}\) is excluded because it has four leaves. The tree of degree sequence \((3,2,1,1,1)\) is excluded because pendant edges at its three leaves, together with the edge joining its two non-leaf vertices, form a matching of size 4. Hence \(R=P_5\). Its endvertices belong to \(S\); neither \(v_2\) nor \(v_4\) may belong to \(S\), since this would again produce a matching of size 4. This leaves \(T_5,T_6\).

The displayed support vertices give matchings and vertex covers of size 3 in \(T_1,T_3,T_4,T_6\). For \(T_2,T_5\), two pendant edges and one internal path edge give a matching of size 3, and a vertex cover of size 3 is immediate. Thus all six families have matching number exactly 3. Finally, the six alternatives are mutually exclusive up to core automorphisms: they are distinguished by the order and isomorphism type of the core and, for a fixed core, by the support pattern modulo its automorphisms.
\end{proof}

Let $\Delta_T(x) = 1-(n-1)x^2+\mu_2x^4-\mu_3x^6$, where $n=|V(T)|$, $\mu_2=m_2(T)$, and $\mu_3=m_3(T)>0$.

Write $\eta_j:=[x^j]N_T(x)$ and define
\begin{equation}\label{eq:nu-three-invariants}
\begin{aligned}
    \rho
      &:=-\frac{\eta_3-\eta_2}{2},
    &
    \sigma
      &:=-\eta_6,
    &
    \chi
      &:=\mu_2-\mu_3-(2n-7),\\
    \kappa
      &:=\eta_5-4\mu_3,
    &
    \lambda
      &:=\eta_5-6\mu_3+4\rho,
    &
    \omega
      &:=\sigma-4(\mu_3-\rho).
\end{aligned}
\end{equation}

\begin{theorem}\label{thm:matching-three-recovery}
Among trees with matching number 3, the polynomial pair \((\Delta_T,N_T)\) determines the tree up to isomorphism.
\end{theorem}

\begin{proof}
Substitution in Proposition~\ref{prop:fixed-core-compression} gives the recognition and recovery certificates in Table~\ref{tab:nu-three-certificates}.

\begin{table}[H]
\centering
\small
\setlength{\tabcolsep}{4pt}
\renewcommand{\arraystretch}{1.62}
\caption{Recognition and recovery certificates for \(\nu(T)=3\).}
\label{tab:nu-three-certificates}
\begin{tabularx}{\textwidth}{c|X|X}
\hline
Type & Recognition condition & Parameter recovery\\
\hline
\(T_1\)
&
\(\sigma=0,\ \eta_5=2\mu_3\)
&
\(
p=\rho,\;
b=\mu_3/\rho,\;
s=n-b-3;
\)
\(\{a,c\}\) are the roots of \(z^2-sz+p\).
\\

\(T_2\)
&
\(\sigma=0,\ \eta_5=4\mu_3,\ \chi=0\)
&
\(
s=n-4,\ p=\mu_3;
\)
\(\{a,d\}\) are the roots of\(z^2-sz+p\).
\\

\(T_5\)
&
\(\sigma=0,\ \eta_5=4\mu_3,\ \chi<0\)
&
\(
s=n-5,\ p=(\mu_3-s)/2;
\)
\(\{a,e\}\) are the roots of
\(z^2-sz+p\).
\\

\(T_4\)
&
\(\sigma>0,\ \lambda=\omega=0\)
&
\(
e_1=n-4,\ e_2=\rho,\ e_3=\mu_3-\rho;
\)
\(\{a,b,c\}\) are the roots of
\(z^3-e_1z^2+e_2z-e_3\).
\\

\(T_3\)
&
\(
\sigma>0,\;
(\lambda,\omega)\neq(0,0),\;
\kappa<0
\)
&
\(
u=-\kappa/2=ab,\;
v=\mu_3-\sigma-u=ad;
\)

\(
    a=uv/\sigma,
    \qquad
    b=\sigma/v,
    \qquad
    d=\sigma/u.
\)
\\

\(T_6\)
&
\(
\sigma>0,\;
(\lambda,\omega)\neq(0,0),\;
\kappa\geq0
\)
&

$c=\mu_2-\rho-n+3, s=n-c-5, p=\rho-cs-2s-1;$

\(\{a,e\}\) are the roots of
\(z^2-sz+p\).
\\
\hline
\end{tabularx}
\end{table}

For completeness, the sign assertions used in the recognition column follow from
\[
\begin{array}{c|c}
\text{Families}&\text{Certificate}\\
\hline
T_1,T_2,T_5&\ \sigma=0,\\
T_1&\eta_5=2\mu_3,\\
T_2,T_5&\eta_5=4\mu_3,\\
T_2&\chi=0,\\
T_5&\chi=-ae<0,\\
T_3&\kappa=-2ab<0,\quad
      \omega-\lambda=-ad(b+2)<0,\\
T_4&\lambda=\omega=0,\\
T_6&\kappa=2c(ae-1)\geq0,\quad
      \lambda=2\bigl(c(a+e-2)+3(a+e)+2\bigr)>0.
\end{array}
\]
Hence the six recognition conditions are mutually exclusive and exhaustive.

The recovery formulas in the last column determine the pendant parameters up to exactly the automorphisms of the corresponding cores: path reflection in \(T_1,T_2,T_5,T_6\), permutation of the three leaves in \(T_4\), and the identity in \(T_3\).  Therefore the tree is uniquely determined up to isomorphism.
\end{proof}

\subsection{Matching number 4}\label{sec:matching-four-families}
We now assume throughout that $\nu(T)=4$. The core reduction of Section~\ref{sec:core-reduction} turns the problem into a finite collection of parametric families. For positive integers \(r,s,t\), let \(S(r,s,t)\) denote the tree obtained by identifying one endpoint of three vertex-disjoint paths of lengths \(r,s,t\). Let $V(R_i)=\{v_1,\ldots,v_{q_i}\}$. We abbreviate a support $\{v_{i_1},\ldots,v_{i_r}\}$ by \(i_1\cdots i_r\). Every parameter corresponding to a displayed support vertex is positive, and all remaining pendant parameters are zero.

\begin{proposition}\label{prop:matching-four-classification}
Let \(T=R(\mathbf{a})\) be the canonical core representation of a tree with \(\nu(T)=4\). Up to an automorphism of \(R\), the pair $\bigl(R,\operatorname{supp}(\mathbf{a})\bigr)$ is one of the 23 entries in Table~\ref{tab:matching-four-families}.
\end{proposition}

\vspace{-20pt}
\begin{table}[H]
\centering
\small
\setlength{\tabcolsep}{3.5pt}
\renewcommand{\arraystretch}{1.25}
\caption{The 10 canonical cores and the 23 admissible
support orbits for \(\nu(T)=4\).  An expression \(ij\) in the edge
column denotes the edge \(v_iv_j\).}
\label{tab:matching-four-families}
\begin{tabularx}{\textwidth}{c|c|X|X}
\hline
Core & Shape & Edge set & Admissible supports\\
\hline
\(R_1\)
&
\(K_{1,4}\)
&
\(15,25,35,45\)
&
\(1234\)
\\

\(R_2\)
&
\(S(3,1,1)\)
&
\(15,25,16,36,46\)
&
\(234,\ 1234\)
\\

\(R_3\)
&
\(S(2,2,2)\)
&
\(15,25,16,36,17,47\)
&
\(234,\ 1234\)
\\

\(R_4\)
&
\(P_7\)
&
\(15,25,16,36,27,47\)
&
\(34,\ 134,\ 1234\)
\\

\(R_5\)
&
\(S(2,1,1)\)
&
\(12,15,35,45\)
&
\(234,\ 1234,\ 2345\)
\\

\(R_6\)
&
\(S(2,2,1)\)
&
\(12,15,35,16,46\)
&
\(234,\ 1234\)
\\

\(R_7\)
&
\(P_6\)
&
\(12,15,35,26,46\)
&
\(34,\ 134,\ 345,\ 1234,\ 1346\)
\\

\(R_8\)
&
\(P_5\)
&
\(12,13,25,45\)
&
\(134,\ 1234,\ 1345\)
\\

\(R_9\)
&
\(K_{1,3}\)
&
\(12,13,14\)
&
\(1234\)
\\

\(R_{10}\)
&
\(P_4\)
&
\(12,13,24\)
&
\(1234\)
\\
\hline
\end{tabularx}
\end{table}

\begin{proof}
By Lemma~\ref{lem:konig}, \(T\) has a minimum vertex cover of cardinality 4. We may choose such a cover \(C=\{c_1,c_2,c_3,c_4\}\) containing no pendant vertex. Indeed, if a pendant vertex \(v\) belongs to a minimum vertex cover \(C\), with unique neighbour \(u\), then \(u\notin C\); otherwise \(C\setminus\{v\}\) would still be a vertex cover. Hence replacing \(v\) by \(u\) preserves both the cardinality and the vertex-cover property. Since \(\nu(T)=4\), the tree \(T\) is not \(K_2\), and thus \(u\) is not pendant. Put $U:=V(T)\setminus C$, $U_2:=\{u\in U:\deg_T(u)\geq2\}$. Since \(C\) is a vertex cover, \(U\) is independent. As \(C\) contains no pendant vertex, \(U\setminus U_2\) is precisely the set of pendant vertices of \(T\). Hence \(R(T)=T[C\cup U_2]\). 

We encode the structure of \(R(T)\) relative to the fixed cover \(C=\{c_1,c_2,c_3,c_4\}\) by two kinds of blocks.  Here the word ``block'' is used only as a convenient encoding device. An edge \(c_ic_j\) will be called a \(D_{ij}\)-block. For each \(u\in U_2\), let \(I(u):=\{i\in\{1,2,3,4\}:u\sim_T c_i\}.\) Since \(U\) is independent, every neighbour of \(u\) lies in \(C\); hence \(N_T(u)=\{c_i:i\in I(u)\}.\) Moreover, \(|I(u)|=\deg_T(u)\ge2.\) We represent such a vertex \(u\) together with its incident edges to
\(C\) by a \(W_{I(u)}\)-block.

Assign weight \(1\) to each \(D\)-block and weight \(|I|-1\) to each \(W_I\)-block. Let \(d\) be the number of \(D\)-blocks and let \(r\) be the number of \(W\)-blocks. Each \(W_I\)-block contributes one vertex and \(|I|\) edges to \(R(T)\), whereas each \(D\)-block contributes one edge. Thus \( |V(R(T))|=4+r, |E(R(T))|=d+\sum_{W_I}|I|.\)
Since \(R(T)\) is a tree, \(|E(R(T))|=|V(R(T))|-1,\) and therefore \(d+\sum_{W_I}(|I|-1)=3.\) Hence the multiset of block weights is one of \(3, 2+1, 1+1+1.\)

If the weight partition is \(3\), there is a single block, necessarily a \(W_I\)-block with \(|I|=4\). Hence \(I=\{1,2,3,4\}\), and \(R(T)\cong K_{1,4}.\)

Suppose next that the weight partition is \(2+1\). The weight-two block must be a ternary \(W\)-block. After relabelling \(c_1,c_2,c_3,c_4\), we may take it to be \(W_{123}\). The remaining weight-one block must involve \(c_4\), since otherwise \(c_4\) would be isolated. Since \(c_1,c_2,c_3\) are equivalent under the symmetry of the ternary block, the remaining block is, up to relabelling, \(W_{14}\) or \(D_{14}\). These two possibilities give $R(T)\cong S(3,1,1)$ or $R(T)\cong S(2,1,1),$ respectively.

Finally, suppose that the weight partition is \(1+1+1\). Then there are exactly three blocks, each of which is either a \(D_{ij}\)-block or a binary \(W_{ij}\)-block.  Construct an auxiliary graph \(H\) on the vertex set \(\{1,2,3,4\}\) by replacing every \(D_{ij}\)- or \(W_{ij}\)-block by the edge \(ij\). The graph \(H\) is simple: if two distinct blocks joined the same pair \(i,j\), then \(R(T)\) would contain two distinct \(c_i\)-\(c_j\) paths, and hence a cycle.  It is connected because \(R(T)\) is connected.  Therefore \(H\) is a connected simple graph on four vertices with three edges, so $H\cong K_{1,3}$ or $H\cong P_4$.

A binary \(W_{ij}\)-block subdivides the corresponding edge \(ij\) of \(H\) once, whereas a \(D_{ij}\)-block leaves it unsubdivided. If \(H\cong K_{1,3}\), subdividing respectively zero, one, two, or three of its edges gives $K_{1,3}$, $S(2,1,1)$, $S(2,2,1)$, $S(2,2,2)$. If \(H\cong P_4\), subdividing zero, one, two, or three of its edges gives \(P_4, P_5, P_6, P_7.\) Combining the three weight partitions and removing the repeated occurrence of \(S(2,1,1)\), we obtain exactly the following ten possible cores:
$K_{1,4}$, $S(3,1,1)$, $S(2,2,2)$, $ P_7$, $ S(2,1,1)$, $S(2,2,1)$, $P_6$, $P_5$, $ K_{1,3}$, $P_4$. These are \(R_1,\ldots,R_{10}\), respectively, in Table~\ref{tab:matching-four-families}. They are pairwise nonisomorphic. Their orders and degree sequences already distinguish all pairs except \(S(3,1,1)\) and \(S(2,2,1)\); in the former, the unique vertex of degree three has exactly one neighbour of degree two, whereas in the latter it has two.

It remains to determine the possible supports. Fix one of the cores \(R_i=R(T)\), and let \(S:=\operatorname{supp}(\mathbf a).\) Every leaf of \(R_i\) belongs to \(S\); otherwise it would also be pendant in \(T\), contradicting \(R_i=R(T)\). Moreover, Lemma~\ref{lem:support-bound} gives $|S|\leq4$. Hence, modulo \(\operatorname{Aut}(R_i)\), every candidate support is obtained from \(L(R_i)\) by adjoining at most \(4-|L(R_i)|\) vertices of \(V(R_i)\setminus L(R_i)\). For each candidate support \(S\), Lemma~\ref{lem:matching-support-formula} gives
\(
\nu(R_i(\mathbf a))
=
\mu_{R_i}(S)
:=
\max_{X\subseteq S}
\bigl(|X|+\nu(R_i-X)\bigr).
\)
In particular, the matching number depends only on the core \(R_i\) and the support \(S\), not on the positive pendant multiplicities. Since \(|S|\leq4\), each value \(\mu_{R_i}(S)\) is obtained by checking at most \(2^4=16\) subsets \(X\subseteq S\) and computing the matching number of the small forest \(R_i-X\). Likewise, the candidate support orbits are obtained by adjoining to \(L(R_i)\) subsets of \(V(R_i)\setminus L(R_i)\) of the permitted sizes and identifying them under \(\operatorname{Aut}(R_i)\). 

The complete list of candidate support orbits, together with their values of \(\mu_{R_i}(S)\), is given in Table~\ref{tab:matching-four-candidates}. The entries with
\(\mu_{R_i}(S)=4\) are precisely the admissible supports.

\vspace{-25pt}
\begin{table}[H]
\centering
\small
\setlength{\tabcolsep}{4pt}
\renewcommand{\arraystretch}{1.35}
\caption{Candidate support orbits and the corresponding values of
\(\mu_{R_i}(S)\).}
\label{tab:matching-four-candidates}
\begin{tabularx}{\textwidth}{c|X}
\hline
Core & Candidate support orbits \(S:\mu_{R_i}(S)\)\\
\hline

\(R_1\)
&
\(1234:\mathbf{4}\)
\\

\(R_2\)
&
\(234:\mathbf{4}\),\
\(1234:\mathbf{4}\),\
\(2345:5\),\
\(2346:5\)
\\

\(R_3\)
&
\(234:\mathbf{4}\),\
\(1234:\mathbf{4}\),\
\(2345:5\)
\\

\(R_4\)
&
\(34:\mathbf{4}\),\
\(134:\mathbf{4}\),\
\(345:5\),\
\(346:5\),\
\(1234:\mathbf{4}\),\
\(1345:5\),\
\(1346:5\),\
\(1347:5\),\
\(3456:5\),\
\(3467:5\)
\\

\(R_5\)
&
\(234:\mathbf{4}\),\
\(1234:\mathbf{4}\),\
\(2345:\mathbf{4}\)
\\

\(R_6\)
&
\(234:\mathbf{4}\),\
\(1234:\mathbf{4}\),\
\(2345:5\)
\\

\(R_7\)
&
\(34:\mathbf{4}\),\
\(134:\mathbf{4}\),\
\(345:\mathbf{4}\),\
\(1234:\mathbf{4}\),\
\(1345:5\),\
\(1346:\mathbf{4}\),\
\(3456:5\)
\\

\(R_8\)
&
\(34:3\),\
\(134:\mathbf{4}\),\
\(234:3\),\
\(1234:\mathbf{4}\),\
\(1345:\mathbf{4}\)
\\

\(R_9\)
&
\(234:3\),\
\(1234:\mathbf{4}\)
\\

\(R_{10}\)
&
\(34:3\),\
\(134:3\),\
\(1234:\mathbf{4}\)
\\

\hline
\end{tabularx}
\end{table}
For example, for \(R_2\) and \(S=234\), the values of \(|X|+\nu(R_2-X)\), as \(X\) ranges over the subsets of \(S\), are \(2,3,3,3,4,4,4,4,\) and hence \(\mu_{R_2}(234)=4\).
For \(S=2345\), taking \(X=S\) already gives \(4+\nu(R_2-S)=5\), and the remaining subsets give no larger value; thus \(\mu_{R_2}(2345)=5\).

The table lists every candidate orbit: for \(R_1,\ldots,R_{10}\), respectively, the numbers of candidates are \(1, 4, 3, 10, 3, 3, 7, 5, 2, 3.\) Retaining exactly the entries with \(\mu_{R_i}(S)=4\) gives, respectively, \(1, 2, 2, 3, 3, 2, 5, 3, 1, 1\) admissible support orbits. These are precisely the supports listed
in the last column of Table~\ref{tab:matching-four-families}. Hence there are \(1+2+2+3+3+2+5+3+1+1=23\) canonical core--support families.
\end{proof}

The preceding argument gives a structural classification only: every tree with matching number 4 belongs to one of the 23 canonical families in Table~\ref{tab:matching-four-families}. To obtain reconstruction from the spectral--walk data, two further issues remain: distinct canonical families must be separated by the data, and within each family the pendant multiplicities must be recovered up to the relevant core automorphisms.

\subsubsection{A finite coefficient encoding}

For a tree \(T\) with matching number 4, write $\eta_j(T):=[x^j]N_T(x)$ and define 
\[\mathcal I(T) := (n, m_2(T), m_3(T), m_4(T), \eta_3(T), \eta_4(T), \dots, \eta_8(T)),\] where \( n = |V(T)| \).

\begin{proposition}\label{prop:I-four-equivalence}
For trees of matching number 4, the finite tuple \(\mathcal I(T)\) and the polynomial pair $\bigl(\Delta_T(x),N_T(x)\bigr)$ determine one another.
\end{proposition}

\begin{proof}
By the matching-polynomial formula, $\Delta_T(x)=1-(n-1)x^2+m_2(T)x^4-m_3(T)x^6+m_4(T)x^8$. Thus the first four entries of \(\mathcal I(T)\) determine
\(\Delta_T\). Since \(\operatorname{rank}A(T)=2\nu(T)=8\), every minor of \(A(T)\) of order greater than eight vanishes. By multilinearity of the determinant, the coefficient of \(x^k\) in any minor of \(I-xA(T)\) is a linear combination of \(k\times k\) minors of \(A(T)\). Hence every cofactor of \(I-xA(T)\) has degree at most eight, and therefore \(\deg N_T\leq 8\). The three coefficients of \(N_T\) not explicitly included in \(\mathcal I(T)\) are $\eta_0=n, \eta_1=2(n-1), \eta_2=-2m_2(T)$. Indeed, from $N_T(x)=\Delta_T(x)F_T(x)$ we have $\eta_0=W_0=n$, $\eta_1=W_1=2(n-1)$, and $\eta_2=W_2-n(n-1)$. For a tree, \(W_2=\sum_{v\in V(T)}\deg_T(v)^2.\) Moreover, since \(T\) has \(n-1\) edges, \(m_2(T)=\binom{n-1}{2}-\sum_{v\in V(T)}\binom{\deg_T(v)}{2}.\) Using \(\sum_v\deg_T(v)=2(n-1)\), we obtain \(W_2=\sum_v\deg_T(v)^2=n(n-1)-2m_2(T).\)
Consequently, \(\eta_2=-2m_2(T)\). Hence \(\mathcal I(T)\) determines \(N_T\).

Conversely, the polynomial pair \((\Delta_T,N_T)\) determines \(\mathcal I(T)\) directly: \(n=N_T(0)\), the coefficients of \(x^4,x^6,x^8\) in \(\Delta_T\) recover
\(m_2(T),m_3(T),m_4(T)\), respectively, and \(\eta_3,\ldots,\eta_8\) are the corresponding coefficients of \(N_T\).
\end{proof}

Thus, throughout the remainder of the matching-number-four case, it suffices to work with the finite invariant \(\mathcal I(T)\). All recovery and separation certificates below will be expressed in terms of its coordinates. For a canonical family \((R_i,S)\), let the positive pendant parameters be ordered according to the increasing order of the vertices in \(S\). Proposition~\ref{prop:fixed-core-compression} expresses $\Delta_{R_i(\mathbf{a})}(x)=\det\bigl(I-xA(R_i)-x^2D_\mathbf{a}\bigr)$ and $N_{R_i(\mathbf{a})}(x)=|\mathbf{a}|\Delta_{R_i(\mathbf{a})}(x)+
(\mathbf1+x\mathbf{a})^{\mathsf T}\operatorname{adj}\bigl(I-xA(R_i)-x^2D_\mathbf{a}\bigr)(\mathbf1+x\mathbf{a})$. Consequently, every entry of \(\mathcal I(R_i(\mathbf{a}))\) is an explicit integer polynomial in the pendant parameters. To prove the matching-number-four case, it remains to establish: (1)injectivity of \(\mathbf{a} \mapsto \mathcal{I}(R_i(\mathbf{a}))\) within each canonical family, modulo \(\operatorname{Aut}(R_i)\); (2) disjointness of the \(\mathcal{I}\)-data of distinct canonical families.

\subsubsection{Reconstruction and separation for matching number 4}\label{sec:matching-four-reconstruction}
For a tree \(T\) with \(\nu(T)=4\), write \(\eta_j(T):=[x^j]N_T(x)\), and omit \(T\) when no ambiguity is possible. By Proposition~\ref{prop:I-four-equivalence}, the
polynomial pair \((\Delta_T,N_T)\) may be replaced, without loss of information, by the finite invariant \(\mathcal I(T)\). Thus, to prove determination in the matching-number-four case, it suffices to show that \(\mathcal I(T)\) distinguishes the canonical families and, within each family, determines the pendant parameters up to the relevant core automorphisms.  In particular, all invariant quantities used below are determined by the spectral--walk data. 

For an admissible pair \((R_i,S)\) in Table~\ref{tab:matching-four-families}, let \(\mathscr F_{i,S}:=\{R_i(\mathbf a):\operatorname{supp}(\mathbf a)=S\}.\) If \(S=\{v_{i_1},\ldots,v_{i_r}\}\), \(i_1<\cdots<i_r,\) the positive pendant parameters are always listed in this order, and all coordinates of \(\mathbf a\) outside \(S\) are set equal to zero.

We use the fixed-core formulas throughout the calculations below. Namely, for \(T=R_i(\mathbf a)\), substitute the labelled adjacency matrix \(A_{R_i}\) and the corresponding pendant vector \(\mathbf a\) into \(M=I-xA_{R_i}-x^2D_{\mathbf a}, \Delta_T(x)=\det M,\) and
\(
N_T(x)
=
|\mathbf a|\det M
+
(\mathbf 1+x\mathbf a)^{\mathsf T}
\operatorname{adj}(M)
(\mathbf 1+x\mathbf a).
\)
The required quantities are then obtained by coefficient extraction:
\[
m_j(T)=(-1)^j[x^{2j}]\Delta_T(x),
\eta_j(T)=[x^j]N_T(x).
\]
Thus, throughout this subsection, ``direct substitution'', ``direct expansion'', and ``direct collection'' mean precisely this substitution and coefficient extraction, followed by elementary polynomial simplification.

\medskip
\noindent\textbf{Internal reconstruction}\mbox{}. We first show that no collision occurs within a canonical family.
\medskip

\begin{proposition}\label{prop:internal-reconstruction}
For each of the twenty-three canonical families
\(\mathscr F_{i,S}\), the map $\mathbf{a}\longmapsto \mathcal I(R_i(\mathbf{a}))$ is injective modulo the action of \(\operatorname{Aut}(R_i)\).
\end{proposition}

\begin{proof}
The explicit recovery certificates are collected in Appendix~\ref{app:recovery-certificates}.  More precisely, Table~\ref{tab:zero-recovery} treats the eight families with
\(\eta_8=0\), while Tables~\ref{tab:positive-recovery-I} and \ref{tab:positive-recovery-II} treat the fifteen families with \(\eta_8>0\). The recovery arguments fall into four elementary types. Throughout the proof, ``the data'' refers to the entries of \(\mathcal I(T)\); by Proposition~\ref{prop:I-four-equivalence}, this is equivalent to
using the polynomial pair \((\Delta_T,N_T)\).

\smallskip
\noindent\emph{\textbf{Symmetric-root recovery.}}
In the families $\mathscr F_{3,234}, \mathscr F_{9,1234}, \mathscr F_{1,1234}, \mathscr F_{3,1234}$, the data determine the elementary symmetric functions of a parameter multiset. The multiset is therefore recovered as the root multiset of the monic polynomial displayed in the corresponding appendix entry. The remaining permutations are precisely those induced by core automorphisms.

\smallskip
\noindent\emph{\textbf{Quadratic-pair recovery.}} In most of the remaining families, the data determine the sum and product of one or two unordered parameter pairs.  Each pair is
therefore recovered as the root multiset of a quadratic polynomial. When two pairs must be matched, the possible mixed terms are $ac+bd$ and $ad+bc$, whose difference is $(a-b)(c-d)$. Thus the data-determined mixed term selects the pairing unless one pair has equal entries; in that case the two pairings are already related by a core reflection.

\smallskip
\noindent\emph{\textbf{Unique positive-root recovery.}}
For \(\mathscr F_{6,1234}\), after the parameter \(a\) has been recovered, the parameter \(b\) is the unique positive root of \(3z^2+Cz-\frac{E}{a}=0,\) because the product of its two roots is \(-E/(3a)<0\). The same argument applies to the quadratic $z^2-Cz-E=0$ in \(\mathscr F_{8,1345}\).

\smallskip
\noindent\emph{\textbf{Candidate elimination.}}
Besides the preceding root-multiset and sign arguments, \(\mathscr F_{6,234}\), the combined \(R_2\), \(R_4\), and \(R_7\) rows, and \(\mathscr F_{7,1346}\) require an additional uniqueness check. For \(\mathscr F_{6,234}\), put \(S=n-6, P=m_2-4S-5, H=m_3-3P-2S-1.\) As recorded in Table~\ref{tab:zero-recovery}, the distinguished
parameter \(x\) satisfies \(x^2-(S-2)x+(m_4-2H)=0, P=x(S-x)+\frac{H}{x}-1.\) Suppose that two distinct positive integers \(x,x'\) satisfied both relations and both yielded positive integral values for the remaining pair. By Vieta's formula, \(x+x'=S-2.\) Subtracting the second relation for \(x\) and \(x'\) gives \(H=2xx'\). Hence, for the candidate \(x\), the remaining pair has sum \(x'+2\) and product \(2x'-1\), and therefore discriminant \((x'+2)^2-4(2x'-1)=(x'-2)^2+4.\) If \(x'=1\), this is \(5\), not a square. If \(x'\ge2\) and it is a square, factoring the difference of squares forces \(x'=2\). Interchanging \(x\) and \(x'\) then gives \(x=2\), contradicting
\(x\ne x'\). Thus \(x\) is uniquely determined. For the combined \(R_2\)-row, the data first recover $\delta=a-b$ and $\{a+b,c+d\}$. If both orientations of this unordered pair were compatible with the data, equality of \(m_3,\eta_4,m_4\) would force $\delta=0$, $|(a+b)-(c+d)|=2$, whereas the two predicted values of \(\eta_7\) would differ by a nonzero quantity. Hence only one orientation is possible. In the exceptional \(R_4\)-case, the unknown product \(p\) is an integer root of $3z^2-(3A+4)z+D=0$. Both roots cannot be integers because their sum is $A+\frac43$. For the combined \(R_7\)-row, the data give at most two candidates for \(x=a+b\). If two distinct candidates \(x,x'\) survived the \(m_3\)-equation, then \(x+x'=S-1\), and their predicted values of \(\eta_5\) would differ by \(-2(R+1)(S-2x-1)=-2(R+1)(x'-x)\neq0.\) Thus the correct candidate is unique. Finally, in \(\mathscr F_{7,1346}\), if both roots of the quadratic for \(a\) produced the known values of the remaining invariants, subtraction would give $R=U$, $(c-1)Z=0$. Since \(Z>0\), this forces \(c=1\), and then \(R=U\) forces the positive parameter \(b\) to vanish, a contradiction. The appendix formulas therefore recover all distinguished parameters and all permissible unordered parameter sets. The remaining ambiguity is exactly the action of the corresponding core automorphism group \(\operatorname{Aut}(R_i)\).
\end{proof}

\medskip
\noindent\textbf{Separation in the zero-\texorpdfstring{\(\eta_8\)}{eta8} stratum}\mbox{}\par

\begin{lemma}\label{lem:zero-eta8-stratum}
A canonical family has \(\eta_8=0\) precisely in the following three rows:
\[
\begin{array}{c|c|c}
\text{Row}&\text{Families}&\eta_7\\
\hline
\mathrm A&
\mathscr F_{3,234},\mathscr F_{5,234},\mathscr F_{6,234}
&-6m_4\\
\mathrm B&
\mathscr F_{7,34},\mathscr F_{7,345},\mathscr F_{8,134}
&-4m_4\\
\mathrm C&
\mathscr F_{9,1234},\mathscr F_{10,1234}
&-2m_4.
\end{array}
\]
Consequently, the row is determined by $-\frac{\eta_7}{2m_4}=3,2,1$. Within each row, distinct families have disjoint \(\mathcal I\)-data.
\end{lemma}

\begin{proof}
Direct substitution gives the displayed values. For every other canonical family, \(\eta_8\) is a nonzero polynomial with nonnegative integer coefficients in the positive pendant parameters, and hence \(\eta_8>0\).

\noindent\textbf{Row A}. Consider \(\mathscr F_{3,234}, \mathscr F_{5,234}, \mathscr F_{6,234}.\) It remains to exclude cross-family collisions within each row. Suppose throughout that two trees from distinct families have the same \(\mathcal I\)-data. For \(\mathscr F_{3,234}\) versus \(\mathscr F_{5,234}\), write \(x,y,z\) for the parameters of the former family and let \(s,p,r\) be their elementary symmetric functions. Write \(X,Y,Z\) for the parameters of the latter family, with \(X\) distinguished, and let \(S,P,R\) be their elementary symmetric functions. Equality of \(n,m_2,m_3,m_4\) gives $S=s+2, P=p+2s+1, R=r+2p+X=3r+2p+s$. Hence \(X=2r+s\), and therefore $Y+Z=S-X=2-2r\le0$, contrary to \(Y,Z>0\).

For \(\mathscr F_{3,234}\) versus \(\mathscr F_{6,234}\), write \(x,y,z\) for the parameters of the former family and let \(s,p,r\) be their elementary symmetric functions.
Write \(X,Y,Z\) for the parameters of the latter family, with \(X\) distinguished, and put \(q=Y+Z, v=YZ,\) so that \(S=X+q, P=Xq+v, R=Xv.\) Equality of \(n,m_2,m_3,m_4\) and \(\eta_6\) yields \(r=(X-1)(s-X-2), 3v(X-1)=7r.\) Since \(r>0\), we have \(X>1\). Cancelling \(X-1\) gives \(3v=7(s-X-2).\) Hence, for some \(k\in\mathbb Z_{>0}\), \(v=7k, s-X-2=3k.\) Substitution into the equality of \(\eta_5\) gives \((3k-5)(X-1)=0,\) which is impossible because \(X>1\) and \(k\in\mathbb Z_{>0}\).

Finally, compare \(\mathscr F_{5,234}\) with \(\mathscr F_{6,234}\). Write \(x,y,z\) for the parameters of \(\mathscr F_{5,234}\), with \(x\) the distinguished parameter, and set \(q=y+z, v=yz,\) \(s=x+q, p=xq+v, r=xv.\) Write \(X,Y,Z\) for the parameters of \(\mathscr F_{6,234}\), with \(X\) distinguished, and set \(Q=Y+Z, V=YZ,\) \(S=X+Q, P=XQ+V, R=XV,\) as in Table~\ref{tab:zero-recovery}. 

Equality of \(n\) and \(m_2\) gives \(S=s-1, p=P+S.\) Equality of \(m_3\), followed by equality of \(\eta_5\), then yields \(R=r-P+S-x-X, S=x+X.\) Consequently,
\(Q=S-X=x, R=r-P,\) $r=P+R=xX+V+XV$. Equality of \(m_4\) also gives $r=2R+XQ=2XV+xX$. Since \(V>0\), comparison yields \(X=1\). Equality of \(-\eta_6\) then reduces to $V+3v=3$, again impossible because \(V,v \in \mathbb{Z}_{>0}\). Thus the three families in row A have pairwise disjoint
\(\mathcal I\)-data.

\noindent\textbf{Row B}. Write \(u,v\) for the parameters of \(\mathscr F_{7,34}\), \(x,y,z\) for those of \(\mathscr F_{7,345}\), and \(a,b,c\) for those of \(\mathscr F_{8,134}\). The data-determined quantity \(D_0:=\eta_6+4m_4\) has the respective values \(0, xyz, ab(c-1).\) Hence \(\mathscr F_{7,34}\) and \(\mathscr F_{7,345}\) are
immediately separated. Suppose that \(\mathscr F_{7,34}\) and \(\mathscr F_{8,134}\) have the same data. Then \(D_0=0\) forces \(c=1\). Equality of \(n\) and \(m_4\) gives
\(u+v=a+b, uv=ab,\) and equality of \(m_2\) then forces \(a=0\), a contradiction. It remains to compare \(\mathscr F_{7,345}\) with \(\mathscr F_{8,134}\). Put
\(\rho=xyz, A=x(y+z).\) Equality of \(m_4\) and \(\eta_6\) would give \(c=1+\frac{\rho}{\rho+A},\) so that \(1<c<2\), contrary to \(c\in\mathbb Z_{>0}\). Thus the three families in row B are pairwise separated.

\noindent\textbf{Row C}. Put \(\Pi:=-\eta_6-4m_4, \Omega:=\frac{\eta_5}{2}-m_3-\Pi.\) Direct substitution gives $\Omega=-2bcd<0$ on $\mathscr F_{9,1234}$, $\Omega=cd>0$ on $\mathscr F_{10,1234}$. Here \(b,c,d\) denote the positive pendant parameters in the notation
of Table~\ref{tab:zero-recovery}. Hence the two families have disjoint \(\mathcal I\)-data.
\end{proof}

\begin{corollary}
\label{cor:zero-eta8-recovery}
If \(T,T'\) have matching number 4, \(\eta_8(T)=\eta_8(T')=0\), and the same spectral-walk data, then \(T\cong T'\).
\end{corollary}

\medskip
\noindent\textbf{Separation in the positive stratum}\mbox{}. We now assume \(E=\eta_8>0\). Retain \(E,K,Q\) from \eqref{eq:EKQ-definition}, and define
\medskip

\begin{equation} \label{eq:sign-sieve-invariants}
\begin{aligned}
    A&:=Q-2E,&
    B&:=Q-2m_4+2E,&
    C&:=Q+2m_4-2E,\\
    D&:=5Q+2m_4-6E,&
    G&:=-\eta_3-2\eta_5-3\eta_6.
\end{aligned}
\end{equation}
These quantities provide a convenient sign sieve for the remaining positive families.

Before applying the sign sieve, two exceptional core types may be set aside. Recall that \(K=-(3\eta_7+6m_4+2\eta_8)/2.\) One has \(K=0\) precisely on \(\mathscr F_{1,1234}\), whereas \(K>0\) on every other positive family; hence the \(R_1\)-family is already globally separated.

The two positive \(R_8\)-families are treated separately. The quantity \(Q=-\eta_7-4m_4\) separates them from all positive families other than the two \(R_5\)-families, and the comparison with \(\mathscr F_{5,1234}\) is separated by \(K/E\)(\(K/E \leq 2\) on \(\mathscr F_{5,1234}\), \(K/E > 2\) on both \(R_8\)-families). Thus only
\(\mathscr F_{5,2345}\ \text{vs.}\ \mathscr F_{8,1345},\mathscr F_{5,2345}\ \text{vs.}\ \mathscr F_{8,1234}\) remain among the \(R_5\)--\(R_8\) comparisons; these are treated
below. We may therefore apply the following sign sieve to the remaining 12 non-\(R_1\), non-\(R_8\) positive families.

\bigskip
\noindent\textit{The global sign sieve}
\medskip

\begin{lemma}\label{lem:preliminary-positive-separation}
On the 12 positive families not involving \(R_1\) or \(R_8\), the quantities \(A,B,C,D,G\) have the signs displayed in Table~\ref{tab:positive-sign-sieve}.
\end{lemma}

\vspace{-16pt}
\begin{table}[H]
\centering
\small
\setlength{\tabcolsep}{5pt}
\renewcommand{\arraystretch}{1.12}
\caption{Sign sieve for the remaining positive families.  An asterisk means that no fixed sign is asserted.}\label{tab:positive-sign-sieve}

\begin{tabular}{c|ccccc}
\hline
Family & \(A\)&\(B\)&\(C\)&\(D\)&\(G\)\\
\hline
\(\mathscr F_{2,234}\)   & \(>0\)&\(<0\)&\(>0\)&\(>0\)&\(*\)\\
\(\mathscr F_{2,1234}\)  & \(<0\)&\(>0\)&\(*\)&\(<0\)&\(*\)\\
\(\mathscr F_{3,1234}\)  & \(*\)&\(>0\)&\(>0\)&\(*\)&\(>0\)\\
\(\mathscr F_{4,34}\)    & \(>0\)&\(<0\)&\(>0\)&\(>0\)&\(*\)\\
\(\mathscr F_{4,134}\)   & \(>0\)&\(*\)&\(>0\)&\(>0\)&\(*\)\\
\(\mathscr F_{4,1234}\)  & \(<0\)&\(>0\)&\(>0\)&\(>0\)&\(*\)\\
\(\mathscr F_{5,1234}\)  & \(<0\)&\(\geq0\)&\(\leq0\)&\(<0\)&\(*\)\\
\(\mathscr F_{5,2345}\)  & \(<0\)&\(<0\)&\(>0\)&\(*\)&\(*\)\\
\(\mathscr F_{6,1234}\)  & \(*\)&\(\geq0\)&\(>0\)&\(*\)&\(*\)\\
\(\mathscr F_{7,134}\)   & \(=0\)&\(*\)&\(>0\)&\(>0\)&\(<0\)\\
\(\mathscr F_{7,1234}\)  & \(<0\)&\(*\)&\(>0\)&\(*\)&\(*\)\\
\(\mathscr F_{7,1346}\)  & \(<0\)&\(*\)&\(>0\)&\(*\)&\(*\)\\
\hline
\end{tabular}
\end{table}

\begin{proof}
For each non-\(*\) entry in the first four columns, substitution of the corresponding explicit formulas for \(E\), \(Q\), and \(m_4\) into the definitions of \(A,B,C,D\) gives an exact polynomial identity in the pendant parameters. More precisely, write every positive pendant parameter as \(a_i=1+x_i, x_i\in\mathbb Z_{\ge0}.\) After substitution and collection, each expression asserted to be strictly positive is a polynomial in the \(x_i\) with nonnegative coefficients and positive constant term, while each expression asserted to be nonnegative has nonnegative coefficients and zero constant term. The negative and nonpositive cases are obtained by applying the same criterion to the negative of the corresponding expression. The only identically zero entry is $A=0$ on $\mathscr F_{7,134}$. This verifies all non-\(*\) entries in the columns \(A,B,C,D\), including the weak inequalities at boundary parameter values.

For illustration, representative computations are
\[
A\big|_{\mathscr F_{2,1234}}
=
-2\bigl(
3abcd+abc+abd+3acd+ac+ad-cd
\bigr)<0,
\]
\[
B\big|_{\mathscr F_{3,1234}}
=
4a\bigl(
2bcd+bc+bd+cd-1
\bigr)>0,
\]
\[
C\big|_{\mathscr F_{5,1234}}
=
-4bcd(a-1)\le0,
\]
and
\[
D\big|_{\mathscr F_{4,134}}
=
2\bigl(
6abc+3ab+2ac+a+10bc+b+c
\bigr)>0.
\]
Here \(a,b,c,d\) denote the positive pendant parameters of the corresponding family. In the first expression, \(3acd-cd=cd(3a-1)>0,\) while the parenthesized factor in the second is positive for \(b,c,d\ge1\). The third expression illustrates a weak inequality, with equality possible at \(a=1\), whereas the fourth is manifestly positive. The remaining non-\(*\) entries in the first four columns are verified in exactly the same way; no further case distinction is required.

It remains only to justify the two asserted signs in the last column \(G\). For \(T\in\mathscr F_{3,1234}\), write $a=1+x, b=1+y, c=1+z, d=1+w$, where \(x,y,z,w\in\mathbb Z_{\ge0}\).  Direct substitution gives
\[
\begin{aligned}
G(T)={}&
12xyzw+18xyz+18xyw+30xy
 +18xzw+30xz+30xw+54x \\
&+30yzw+24yz+24yw+16y
 +24zw+16z+16w+24,
\end{aligned}
\]
which is strictly positive. For \(T\in\mathscr F_{7,134}\), with positive pendant parameters \(a,b,c\), direct expansion gives
\(
    G(T)
    =
    -7ab-3ac-a-12bc-10b-10c+10.
\)
Since \(a,b,c\ge1\), $G(T)\le -33<0$.

Hence every asserted sign in Table~\ref{tab:positive-sign-sieve} holds.
\end{proof}

\medskip
\noindent\textit{The two difficult \texorpdfstring{\(R_5\)--\(R_8\)}{R5--R8} comparisons}
\medskip

\begin{lemma}\label{lem:R5-R8-1345}
The families $\mathscr F_{5,2345}$ and $\mathscr F_{8,1345}$ are separated by the spectral-walk data.
\end{lemma}

\begin{proof}
Suppose that $T=R_5(a,b,c,d)$, $T'=R_8(\alpha,\beta,\gamma,\delta)$ have equal data, with supports \(2345\) and \(1345\), respectively. Put \(s=b+c, q=bc,\)
and
\[
u=\alpha+\delta,\qquad
v=\alpha\delta,\qquad
r=\beta+\gamma,\qquad
w=\beta\gamma,\qquad
\ell=\alpha\beta+\gamma\delta.
\] Since $Q(T)=2bc(a-d)$, $Q(T')=0$, we have \(a=d=:t\). Equality of \(K/E\), \(E\), and \(n\) gives
$v=\frac{tu}{2}, w=\frac{2tq}{u}, r=2t+s-u$. Since \(\alpha,\delta>0\), \(u^2=(\alpha+\delta)^2\ge4\alpha\delta=4v=2tu,\) and hence \(u\ge2t\). If \(u=2t\), direct substitution into the two formulas for \(m_2\) gives \(m_2(T)-m_2(T')=-1,\) a contradiction, so write $u=2t+h$, $h\in\mathbb Z_{>0}$. Equality of the two values of \(\eta_6\) determines the remaining mixed term: \(\ell=t^2s+2tq+ts+4q-vr.\) After substituting
\(
u=2t+h,
v=\frac{tu}{2},
w=\frac{2tq}{u},
r=2t+s-u
\)
and the above expression for \(\ell\), the equalities of \(m_2\) and \(m_3\) are linear in \(q\) and \(s\). Solving them gives
\[
q=
\frac{(2t+h)(th+2)}
     {2(4t^2+2th+8t+3h)}
\]
and
\[
s=
\frac{
6t^3h+7t^2h^2+16t^2h-4t^2
+2th^3+15th^2+6th-8t
+3h^3+3h^2-2h
}{
(4t^2+2th+8t+3h)h
}.
\]
All denominators are strictly positive since \(t,h>0\). Hence \(q\) and \(s\) are uniquely determined. 

Substitution in the two expressions for the total-walk count \(W_3\) gives \[W_3(T)-W_3(T')=\frac{P(t,h)}{(4t^2+2th+8t+3h)h},\] where
\[
P(t,h)=(-t^3-t^2+3t+6)h^3+(-2t^4-4t^3+12t^2+24t-2)h^2+(8t^3+16t^2-4t-4)h-8t(t+2).
\]

Moreover, \(P(1,h)=7h^3+28h^2+16h-24>0, P(2,h)=30h^2+116h-64>0.\) For \(t=3\), \(P(3,h)=-21h^3-92h^2+344h-120.\) Its values at \(h=1,2\) are \(111\) and \(32\), while for
\(h\ge3\), \(21h^3+92h^2>344h,\) so \(P(3,h)<0\). Similarly, \(P(4,h)=-62h^3-482h^2+748h-192,\) with \(P(4,1)=12\), whereas for \(h\ge2\), \(62h^3+482h^2>748h,\) and hence \(P(4,h)<0\). If \(t\ge5\), write $P(t,h)=A_3h^3+A_2h^2+A_1h+A_0$. Then \(A_3<0\), \(A_2<0\), \(A_0<0\), and \(A_3+A_2+A_1=-t(2t^3-3t^2-27t-23)<0.\) Since \(h\ge1\), \(A_3h^3+A_2h^2+A_1h\le h(A_3+A_2+A_1)<0,\) and hence \(P(t,h)<0\). Thus \(P(t,h)\neq0\), contradicting equality of \(W_3\).
\end{proof}

\begin{lemma}\label{lem:R5-R8-1234}
The families $\mathscr F_{5,2345}$ and $\mathscr F_{8,1234}$ are separated by the spectral-walk data.
\end{lemma}

\begin{proof}
Assume that \(T=R_5(a,b,c,d)\in\mathscr F_{5,2345}, T'=R_8(\alpha,\beta,\gamma,\delta)\in\mathscr F_{8,1234}\) have equal spectral-walk data. Write the \(R_5\)-parameters as \(a,b,c,d\), and put \(S=b+c, P=bc.\) Write the \(R_8\)-parameters as \(\alpha,\beta,\gamma,\delta\). Equality of \(K/E\) and \(Q\) gives \(d=2\beta, \frac1a=\frac1d+\frac1\delta.\) Equality of \(E\) and \(n\) further gives
\(
    \alpha\gamma=\frac{2aP}{\delta},
    \alpha+\gamma=a+S+\frac d2-\delta.
\)

We now put the reciprocal relation into integral form. Write \(d=g\rho, \delta=g\sigma, \gcd(\rho,\sigma)=1.\) Then \(a=\frac{g\rho\sigma}{\rho+\sigma}.\) Since \(\gcd(\rho\sigma,\rho+\sigma)=1,\) integrality of \(a\) implies \(g=k(\rho+\sigma)\) for some \(k\in\mathbb Z_{>0}\). Hence
\begin{equation}\label{eq:hard-param}
\begin{split}
a&=k\rho\sigma,\qquad
d=k\rho(\rho+\sigma),\qquad
\beta=\frac{k\rho(\rho+\sigma)}2,\qquad
\delta=k\sigma(\rho+\sigma),\\
\alpha\gamma &=\frac{2\rho P}{\rho+\sigma},\qquad
\alpha+\gamma =S+\frac{k(\rho-\sigma)(\rho+2\sigma)}2.
\end{split}
\end{equation}

It remains to exclude \(\rho=\sigma\). In that case coprimality forces \(\rho=\sigma=1.\) Then \(a=k, d=\delta=2k, \beta=k,\) and \(\alpha+\gamma=S, \alpha\gamma=P.\)
Since \(S=b+c\) and \(P=bc\), it follows that \(\{\alpha,\gamma\}=\{b,c\}.\) Equality of \(m_2\) then forces $\alpha=3k+1$. Substitution into the equality of \(\eta_6\) gives
\(0=-(4k^2+11k+3)\gamma+k^2+k<0,\) a contradiction. Hence \(\rho\ne\sigma.\)

Put $t:=\frac{\rho}{\sigma}, u:=k\sigma^2, L:=\frac{u(t-1)(t+2)}2$. Thus $a=tu, d=t(t+1)u, \beta=\frac{t(t+1)u}{2},\delta=(t+1)u$, and $\alpha+\gamma=S+L, \alpha\gamma=\frac{2tP}{t+1}$. 

Let \(D_2:=m_2(T)-m_2(T'), D_3:=m_3(T)-m_3(T'), D_6:=\eta_6(T)-\eta_6(T').\) After substituting the preceding parametrization, the combinations \(D_3-(\delta+1)D_2, D_6+\beta(\delta+1)D_2\) give a \(2\times2\) affine system in \(S\) and \(P\). Its coefficient determinant is $D_0=\frac{4(t+1)^2u}{\sigma^2}D(k,\rho,\sigma)$, where
\[
\begin{aligned}
D(k,\rho,\sigma) ={}& k^3 \rho^2 \sigma^2 (\rho^2 - \sigma^2)^2 
- 2k^2 \rho \sigma^2 (\rho + 2\sigma)(\rho^2 - \sigma^2) \\
& + k(-\rho^4 + 2\rho^3\sigma - 15\rho^2\sigma^2 - 10\rho\sigma^3 + 16\sigma^4) 
- 2\rho^2 + 6\rho\sigma - 16\sigma^2.
\end{aligned}
\]
We claim that \(D(k,\rho,\sigma)\ne0\). If \(\rho<\sigma\), put \(X=k-1, Y=\sigma-\rho-1, Z=\rho-1.\) Direct expansion gives \(D(k,\rho,\sigma)=221+P_{<}(X,Y,Z),\)
where \(P_{<}\) has nonnegative integer coefficients and zero constant term. Hence \(D(k,\rho,\sigma)>0\). Suppose \(\rho>\sigma\), and write \(\rho=\sigma+h\). For \(k\ge2\), direct expansion shows that
\[
    \left.\frac{\partial D}{\partial k}\right|_{k=2}>0,
    \quad
    \frac{\partial^2D}{\partial k^2}>0,
     (k\ge2).
\]
Hence \(D\), viewed as a cubic in \(k\), is strictly increasing on \(k\ge2\). Its value at \(k=2\) is positive except when \((h,\sigma)=(1,1)\); in that exceptional case  $D(2,2,1)=-44, D(3,2,1)=336$, so no integral \(k\ge2\) gives \(D=0\). 

For \(k=1\) and \(h\ge2\), put \(p=h-2, q=\sigma-1,\) so that \(p,q\in\mathbb Z_{\ge0}\). Direct expansion gives
\[
D(1,\sigma+h,\sigma)
=
144+780p+872q+P_1(p,q),
\]
where \(P_1(p,q)\) is a polynomial with nonnegative integer coefficients and zero constant and linear terms. Hence \(D(1,\sigma+h,\sigma)>0, (h\ge2).\)

It remains to consider \(h=1\). We have \(D(1,2,1)=-88, D(1,3,2)=-483.\) For \(\sigma\ge3\), write \(\sigma=z+3\), where \(z\in\mathbb Z_{\ge0}\). Then
\[
D(1,\sigma+1,\sigma)={}4z^6+72z^5+523z^4+1912z^3+3518z^2+2640z+96>0.
\]
Thus \(D\ne0\) also when \(k=1\). Therefore \(D(k,\rho,\sigma)\ne0\) in all cases, and hence \(D_0\ne0\). Thus \(S\) and \(P\) are uniquely determined.

The coefficient of \(\gamma\) in \(D_2\) is \(-1\). Hence the equality \(D_2=0\) determines \(\gamma\) uniquely. Put \(\Lambda(t,u):=\frac{D_0}{4u(t+1)^2}.\) Since \(D_0\ne0\), we have \(\Lambda(t,u)\ne0\).
Solving the linear equation for \(\gamma\), and using \(\gamma^2-(S+L)\gamma+\frac{2tP}{t+1}=0\) to eliminate \(\gamma^2\), direct substitution into the common third total-walk count \(W_3\) yields \[W_3(T)-W_3(T')=\frac{H(t,u)}{u(t+1)\Lambda(t,u)}, \Lambda(t,u)\neq0,\]where
\[
\begin{aligned}
H(t,u)={}&
t^3(t-1)^2(t+1)^4(t+2)u^6
+t^2(t-1)(t+1)^3(4t^3+7t^2-13t-10)u^5\\
&-t(t+1)^2(3t^5-16t^4+5t^3+90t^2-10t-40)u^4\\
&-2(t+1)^2(t^5+12t^4-9t^3+28t^2+36t-16)u^3\\
&-4(3t^5+21t^4+27t^3+23t^2+18t+8)u^2
 -8(3t^3+12t^2+9t+4)u-16(t+2).
\end{aligned}
\]
To prove \(H\neq0\), clear powers of \(\sigma\) and put $W(k,\rho,\sigma):=\sigma\,H\!\left(\frac{\rho}{\sigma},k\sigma^2\right)$. We treat the three unbounded parameter ranges separately. In each range we translate the positive-integer constraints to a nonnegative orthant by introducing slack variables. Thus, when
\(\rho<\sigma\), we write $X=k-1, Y=\sigma-\rho-1, Z=\rho-1$, so that \((k,\rho,\sigma)=(X+1,Z+1,Y+Z+2), X,Y,Z\in\mathbb Z_{\ge0}.\)
Similarly, for \(\rho>\sigma\) and \(k\ge2\), we write $X=k-2, Y=\rho-\sigma-1, Z=\sigma-1$, whereas for \(k=1\) and \(\rho-\sigma\ge2\) we write $Y=\rho-\sigma-2, Z=\sigma-1$. The constants below are simply the values of \(W\) at the corresponding base points: \(W(1,1,2)=19324,W(2,2,1)=136800,W(1,3,1)=206336.\)
After subtracting these base values, direct expansion shows that each of \(W(X+1,Z+1,Y+Z+2)-19324, W(X+2,Y+Z+2,Z+1)-136800, W(1,Y+Z+3,Z+1)-206336\) is a polynomial with nonnegative integer coefficients and zero constant term in the displayed variables.

These cover respectively \(\rho<\sigma\), \(\rho>\sigma,\ k\geq2\), and \(\rho-\sigma\geq2,\ k=1\). In the remaining case \(k=1,\rho=\sigma+1\), the value is negative at \(\sigma=1\), while after writing \(\sigma=z+2\) it becomes a polynomial in \(z\) with strictly positive coefficients.  Therefore \(W\), and hence \(H\), never vanishes. This contradicts equality of \(W_3\).
\end{proof}

\medskip
\noindent\textit{Remaining cross-family comparisons}
\medskip

For brevity, put $F_2=\mathscr F_{2,1234}, F_3=\mathscr F_{3,1234}, F_4=\mathscr F_{4,1234}, F_5=\mathscr F_{5,1234}, F_6=\mathscr F_{6,1234}, F_{7a}=\mathscr F_{7,1234}, F_{7b}=\mathscr F_{7,1346}$. Combining the global sign sieve with the cross-family separations established above leaves exactly the following nine unresolved comparisons:
\[
\begin{aligned}
F_2&\text{ vs. }F_3,  F_5,  F_{7a},  F_{7b},\\
F_3&\text{ vs. }F_4,F_{7b},\\
F_4&\text{ vs. }F_6,F_{7b},\\
F_6&\text{ vs. }F_{7b}.
\end{aligned}
\]
Indeed, any other pair is either already covered by a preceding separation result or has disjoint ranges in at least one of the columns \(A,B,C,D,G\) of Table~\ref{tab:positive-sign-sieve}.

\vspace{-8pt}
\begin{table}[H]
\centering
\small
\setlength{\tabcolsep}{3pt}
\renewcommand{\arraystretch}{1.33}
\caption{Final linear separation certificates.  After replacing each positive parameter by \(1+x\), the displayed \(L\) is a polynomial with nonnegative coefficients and the stated positive lower bound on the first family; its negative has the same property on the second.}
\label{tab:final-linear-separators}
\begin{tabularx}{\textwidth}{c|X|c}
\hline
Pair & Separator \(L\) & Bounds\\
\hline
\(F_2,F_3\)
&
\(2m_3+4m_4-\eta_3-2\eta_4-\eta_6+\eta_7\)
&
\(L\geq14,\ -L\geq6\)
\\

\(F_2,F_{7a}\)
&
\(-n+m_3+\eta_7+3\eta_8\)
&
\(L\geq5,\ -L\geq4\)
\\

\(F_2,F_{7b}\)
&
\(-m_2+2m_3+4m_4+\eta_6+\eta_7+3\eta_8\)
&
\(L\geq2,\ -L\geq2\)
\\

\(F_3,F_4\)
&
\(-m_2+8m_3-9m_4-\eta_4-\eta_5-\eta_6-\eta_7-\eta_8\)
&
\(L\geq12,\ -L\geq8\)
\\

\(F_3,F_{7b}\)
&
\(3-2n+5m_3-6m_4-\eta_5-\eta_7\)
&
\(L\geq1,\ -L\geq6\)
\\

\(F_4,F_6\)
&
\(1+n-4m_2+4m_3+2\eta_4-\eta_5
  +3\eta_6-\eta_7+3\eta_8\)
&
\(L\geq1,\ -L\geq11\)
\\

\(F_4,F_{7b}\)
&
\(-n-m_2+7m_3-10m_4-\eta_5+\eta_6-2\eta_7+\eta_8\)
&
\(L\geq2,\ -L\geq1\)
\\

\(F_6,F_{7b}\)
&
\(-3-2m_2+8m_3-10m_4-\eta_3-2\eta_4-\eta_5
 -2\eta_6-\eta_7-2\eta_8\)
&
\(L\geq9,\ -L\geq1\)
\\
\hline
\end{tabularx}
\end{table}

\begin{lemma}\label{lem:eight-linear-separators}
The 8 pairs in Table~\ref{tab:final-linear-separators} are separated by the displayed data-determined linear forms.
\end{lemma}

\begin{proof}
Substitute the appropriate fixed-core expressions. In each row, write the four positive parameters as
\[
a=1+x,\qquad b=1+y,\qquad c=1+z,\qquad d=1+w,
\]
where \(x,y,z,w\in\mathbb Z_{\ge0}\). The variables \(x,y,z,w\) are local to each family evaluation. Direct collection gives
\[
L|_{\mathrm{first}}=c_+ + P_+(x,y,z,w),\quad
-L|_{\mathrm{second}}=c_- + P_-(x,y,z,w),
\]
where \(P_+\) and \(P_-\) have nonnegative coefficients and zero constant term, and \(c_+,c_->0\) are the constants displayed in the last column. For example, for the pair \(F_2,F_3\), \[L|_{F_2}={}2wxy+2wxz+5wx+2wyz+wy+w + 2xyz+5xz+2x+yz+4y+z+14,\] whereas
\[
-L|_{F_3}
=
2(3wy+3wz+5w+3yz+5y+5z+3).
\]
Thus \(L|_{F_2}\ge14\) and \(-L|_{F_3}\ge6\).
The other seven rows are verified by the same direct
positive-coefficient collection. Hence
\(L|_{\mathrm{first}}\ge c_+>0, L|_{\mathrm{second}}\le -c_-<0.\) Since \(L\) is determined by the spectral-walk data, the two families cannot have equal data.
\end{proof}

The ninth comparison requires a short integer argument.

\begin{lemma}\label{lem:F2-F5-separation}
The families $F_2=\mathscr F_{2,1234}$ and $F_5=\mathscr F_{5,1234}$ are separated by the spectral-walk data.
\end{lemma}

\begin{proof}
Suppose that \(T=R_2(a,b,c,d)\in F_2, T'=R_5(\alpha,\beta,\gamma,\delta)\in F_5\) have equal data. Put \(s=c+d\) and \(p=cd\). Guided by the fixed-core expressions, we seek data-determined linear combinations that vanish identically on \(F_5\) and factor simply on \(F_2\). The following two combinations have this property:
\[
\begin{aligned}
\Phi_1
 &:=
128-32n+24m_2-12m_3+6m_4
 +4\eta_3+2\eta_5+\eta_7,\\
\Phi_2
 &:=
-384+96n-66m_2+10m_3-9\eta_3
 +6\eta_4-2\eta_5+4\eta_6+4\eta_8.
\end{aligned}
\]
Direct substitution gives $\Phi_1(T')=\Phi_2(T')=0$, whereas $\Phi_1(T)=-2b(a-1)(p-s)$ and $\Phi_2(T)=4b\bigl((a-3)(p-s)+2(a-1)\bigr)$. Equality therefore gives
$(a-1)(p-s)=0$, $(a-3)(p-s)+2(a-1)=0$. If \(a=1\), the second equation gives \(p=s\); if \(p=s\), it gives \(a=1\). Hence \(a=1, p=s.\) Thus $cd=c+d$, $(c-1)(d-1)=1$, and positivity gives \(c=d=2\). 

Direct substitution also gives \(\eta_3(T)+4m_2(T)=2(3a+b+2c+2d+3),\) whereas \(\eta_3(T')+4m_2(T')=2(\alpha\beta+2\alpha+\beta+2\gamma+2\delta+2).\) Equality of the orders gives \(\alpha+\beta+\gamma+\delta=a+b+c+d+1.\) Substituting this into the equality of the preceding two expressions and cancelling common terms gives
\(\beta(\alpha-1)=a-b-1.\) Since \(a=1\), this becomes \(\beta(\alpha-1)=-b<0.\) This is impossible since \(\alpha,\beta\in\mathbb Z_{>0}\), so \(\beta(\alpha-1)\ge0\).
\end{proof}

\begin{theorem}\label{thm:matching-four-recovery}
Let \(T,T'\) be trees with $\nu(T)=\nu(T')=4$. If \(T\) and \(T'\) have the same spectral-walk data, then $T\cong T'$.
\end{theorem}

\begin{proof}
By Proposition~\ref{prop:spectral-walk-polynomials}, the two trees have the same pair \((\Delta,N)\), and hence the same tuple
\(\mathcal I\).
If \(\eta_8=0\), apply Corollary~\ref{cor:zero-eta8-recovery}. Suppose that \(\eta_8>0\). Proposition~\ref{prop:internal-reconstruction} gives injectivity
inside every positive canonical family.  The \(R_1\)-family is separated by \(K=0\); the preliminary certificates and Lemmas~\ref{lem:R5-R8-1345}--\ref{lem:R5-R8-1234} remove the first cross-family comparisons. Table~\ref{tab:positive-sign-sieve} removes all but nine of the remaining comparisons. These nine are settled by Lemmas~\ref{lem:eight-linear-separators} and \ref{lem:F2-F5-separation}. Thus \(T,T'\) lie in the same canonical family and have the same pendant vector modulo the relevant core automorphism. Hence \(T\cong T'\).
\end{proof}

\begin{proof}[Proof of Theorem~\ref{thm:main-rigidity}]
\leavevmode\\
Let \(T\) and \(T'\) have the same spectral--walk data and satisfy $\nu(T),\nu(T')\le 4$. Since \(T\) and \(T'\) are cospectral, $\Delta_T=\Delta_{T'}$. By Lemma~\ref{lem:sachs-forest}(Sachs' theorem), $\nu(T)=\frac12\deg\Delta_T=\frac12\deg\Delta_{T'}=\nu(T')$. If the common matching number is at most 2, apply Proposition~\ref{prop:matching-at-most-two}; if it is 3, apply Theorem~\ref{thm:matching-three-recovery}; and if it is 4, apply Theorem~\ref{thm:matching-four-recovery}. In either case \(T\cong T'\).
\end{proof}

% =========================================================
\section{Sharpness: a family with matching number 5}\label{sec:sharpness-section}
% =========================================================

We now prove Theorem~\ref{thm:sharpness}. Let $P_7$ be the path with vertices $v_1,v_2,\ldots,v_7$ in this order. For $\mathbf a=(a_1,\ldots,a_7)\in\mathbb{Z}_{\ge 0}^7$, write $P_7(\mathbf a)$ for the tree obtained from $P_7$ by attaching $a_i$ pendant vertices to $v_i$. The first collision found in the computation for $n=17$ and $\nu=5$ is
 $P_7(2,0,0,1,6,0,1)$ and $P_7(1,0,0,1,0,6,2)$. This suggests the following family. 
 
For every integer $m\ge 1$, define $T_m=P_7(2,0,0,1,m,0,1)$, and $T_m'=P_7(1,0,0,1,0,m,2)$.

\begin{proof}[Proof of Theorem~\ref{thm:sharpness}]
Let $R=P_7$, and let $A_R$ be the adjacency matrix of $R$ with respect to the vertex ordering $v_1,\ldots,v_7$. For $\mathbf c=(c_1,\ldots,c_7)$, put $D_{\mathbf c}=\diag(c_1,\ldots,c_7)$, and $M_{\mathbf c}(x) = I_7 - xA_R - x^2D_{\mathbf c}$. Set $\mathbf a=(2,0,0,1,m,0,1)$ and $\mathbf b=(1,0,0,1,0,m,2)$. Then $T_m=P_7(\mathbf a)$ and $T_m'=P_7(\mathbf b)$. 

By Proposition~\ref{prop:fixed-core-compression}, $\det(I-xA(T_m)) = \det M_{\mathbf a}(x)$ and $\det(I-xA(T_m')) = \det M_{\mathbf b}(x)$. A direct computation of the tridiagonal determinants gives $\det M_{\mathbf a}(x) = \det M_{\mathbf b}(x) = \Delta_m(x)$, where
\[
    \Delta_m(x) = (1-2x^2)\left[1-(m+8)x^2+(6m+18)x^4-(9m+9)x^6+(2m+1)x^8\right].
\]
Therefore $\Spec(T_m)=\Spec(T_m')$.

Next, by Proposition~\ref{prop:fixed-core-compression}, 
\[F_{T_m}(x) = |\mathbf a| + (\one+x\mathbf a)^{\mathsf T} M_{\mathbf a}(x)^{-1} (\one+x\mathbf a), \quad 
\text{and} \quad
F_{T_m'}(x) = |\mathbf b| + (\one+x\mathbf b)^{\mathsf T} M_{\mathbf b}(x)^{-1} (\one+x\mathbf b).
\] 
Since $|\mathbf a|=|\mathbf b|=m+4$, direct substitution gives
\[F_{T_m}(x) = F_{T_m'}(x) = \frac{N_m(x)}{\Delta_m(x)},\] where
\[
\begin{aligned}
    N_m(x) ={}& m+11 +(2m+20)x -(16m+68)x^2 \\
    &-(26m+118)x^3 +(62m+126)x^4 +(86m+204)x^5 \\
    &-(81m+69)x^6 -(94m+98)x^7 +(27m+11)x^8 \\
    &+(20m+12)x^9 -2mx^{10}.
\end{aligned}
\]
Thus $F_{T_m}(x)=F_{T_m'}(x)$, which is equivalent to $W_k(T_m)=W_k(T_m')$ for all $k\ge 0$. It remains to show that the two trees are non-isomorphic and have matching number 5.

For $T_m$, the five edges $v_1\ell_1,  v_2v_3,  v_4\ell_4,  v_5\ell_5,  v_6v_7$ form a matching, where $\ell_i$ denotes a pendant vertex adjacent to $v_i$. Hence $\nu(T_m)\ge 5$. On the other hand, $\{v_1,v_3,v_4,v_5,v_7\}$ is a vertex cover of $T_m$, so $\nu(T_m)\le 5$. Therefore $\nu(T_m)=5$.

Similarly, for $T_m'$, the five edges $v_1\ell_1,  v_2v_3,  v_4\ell_4,  v_6\ell_6,  v_7\ell_7$ form a matching, while $\{v_1,v_3,v_4,v_6,v_7\}$ is a vertex cover. Therefore $\nu(T_m')=5$.

Finally, the branching vertices distinguish the two trees. Here a branching vertex means a vertex of degree at least 3. In $T_m$, the branching vertices are $v_1, v_4, v_5$, and their pairwise distances form the multiset $\{1,3,4\}$. In $T_m'$, the branching vertices are $v_4, v_6, v_7$, and their pairwise distances form the multiset $\{1,2,3\}$. Since this multiset is an isomorphism invariant, we have $T_m\not\cong T_m'$. This completes the proof.
\end{proof}

\begin{corollary}
The bound $\nu\le 4$ in Theorem~\ref{thm:main-rigidity} is sharp.
\end{corollary}

\begin{proof}
Theorem~\ref{thm:sharpness} gives an infinite family of non-isomorphic pairs of trees with matching number 5 having the same spectrum and the same total-walk sequence.
\end{proof}

% =========================================================
% Acknowledgements
% =========================================================

\section*{Acknowledgements}

I thank Wenjian Zhang for his assistance with the programming and computational aspects of this work. This work was supported by the High-level Talent Research Start-up Fund of West Anhui University (Project No. WGKQ2021072).

% =========================================================
% Declaration of generative AI and AI-assisted technologies in the writing process
% =========================================================
\section*{Declaration of Use of AI Tools}
During this research, the author used ChatGPT (OpenAI) to explore proof strategies and assist with Python code development. Python programs were used to conduct systematic searches over the tree families considered; these searches identified the collision pairs for matching number five, which were subsequently verified and proved mathematically by the author. AI-assisted tools were also used for formula checking, proofreading, and language editing. All AI-generated outputs were independently verified by the author, who assumes full responsibility for the content of this work.

\appendix
\counterwithin{table}{section}
\counterwithin{equation}{section}

\section{Recovery certificates for the canonical families}\label{app:recovery-certificates}

\renewcommand{\theequation}{S\arabic{equation}}
\setcounter{equation}{0}

This appendix records the explicit recovery identities used in the proof of Proposition~\ref{prop:internal-reconstruction}. For the positive families, we use the data-determined quantities
\begin{equation}
 E:=\eta_8,
    \qquad
    K:=-\frac{3\eta_7+6m_4+2\eta_8}{2},
    \qquad
    Q:=-\eta_7-4m_4.
    \label{eq:EKQ-definition}
\end{equation}
   
All parameters corresponding to displayed support vertices are positive, except for the optional parameters explicitly allowed to vanish in the combined rows.

\subsection{Families with \texorpdfstring{\(\eta_8=0\)}{eta8 = 0}} \label{app:zero-recovery}

The recovery certificates for the 8 families satisfying \(\eta_8=0\) are given in Table~\ref{tab:zero-recovery}.

\begingroup
\small
\setlength{\tabcolsep}{4pt}
\renewcommand{\arraystretch}{1.18}

\clearpage
\begin{longtable}{@{}L{0.205\textwidth}L{0.755\textwidth}@{}}
\caption{Internal recovery certificates when \(\eta_8=0\).}
\label{tab:zero-recovery}\\

\toprule
Family & Recovery certificate\\
\midrule
\endfirsthead

\multicolumn{2}{c}{%
\tablename\ \thetable\ continued from the previous page%
}\\
\toprule
Family & Recovery certificate\\
\midrule
\endhead

\midrule
\multicolumn{2}{r}{Continued on the next page}\\
\endfoot

\bottomrule
\endlastfoot

\(\mathscr F_{3,234}\)
&
For pendant parameters \(a,b,c\), recover their elementary symmetric functions 
\(s:=a+b+c=n-7, p:=ab+ac+bc=m_2-5s-9, r:=abc=m_3-4p-5s-4.\)

Then \(\{a,b,c\}\) is the root multiset of $z^3-sz^2+pz-r$.
\rule[-1.5mm]{0pt}{1.5mm}\\

\(\mathscr F_{5,234}\)
&
Write $T=R_5(0,x,y,z,0)\in\mathscr F_{5,234},\quad x,y,z>0$.

Recover
\(
s:=x+y+z=n-5,\quad
p:=xy+xz+yz=m_2-3s-2,\quad
r:=xyz=m_4.
\)
Then $x=r+2p+s-m_3,\quad y+z=s-x,\quad yz=r/x$,

so \(x\) and the unordered pair \(\{y,z\}\) are recovered.
\rule[-1.5mm]{0pt}{1.5mm}\\

\(\mathscr F_{6,234}\)
&
Write $T=R_6(0,x,y,z,0,0)\in\mathscr F_{6,234}, \quad x,y,z>0$, and set
\(
S:=n-6=x+y+z,
\qquad
P:=m_2-4S-5=xy+xz+yz,
\)
\(H:=m_3-3P-2S-1=xyz+x=x(yz+1).\)

Then the distinguished parameter \(x\) is the unique positive integral solution of
\(\displaystyle
x^2-(S-2)x+(m_4-2H)=0,
\qquad
P=x(S-x)+\frac{H}{x}-1.
\)
Finally, $y+z=S-x,\quad yz=\frac{H}{x}-1$, so the unordered pair \(\{y,z\}\) is recovered.
\rule[-1.5mm]{0pt}{1.5mm}\\

\(\mathscr F_{7,34}\)
&
Let \(x,y>0\) be the pendant multiplicities at \(v_3,v_4\). Since the reflection of \(R_7=P_6\) interchanges them, it suffices to recover \(\{x,y\}\). 

From $x+y=n-6,\quad xy=m_4$, the pair \(\{x,y\}\) is the root multiset of $t^2-(n-6)t+m_4$.
\rule[-1.5mm]{0pt}{1.5mm}\\

\(\mathscr F_{7,345}\)
&
Write $T=R_7(0,0,x,y,z,0)\in\mathscr F_{7,345}, \quad x,y,z>0$.

Define
\(
    s:=n-6=x+y+z,\qquad
    \rho:=\eta_6+4m_4=xyz,
\)
\(
    A:=m_4-2\rho=x(y+z),\qquad
    C:=m_2-A-4s-6=yz-z,
\)
\(K_0:=\eta_4-3\rho-8A-8s-5C-4=xz.\)

Then
\(\displaystyle
    xy=A-K_0,\qquad
    z=\frac{\rho}{A-K_0},\qquad
    x=\frac{K_0}{z},\qquad
    y=\frac{A-K_0}{x}.
\)

Thus \((x,y,z)\) is uniquely recovered.
\rule[-1.5mm]{0pt}{1.5mm}\\

\(\mathscr F_{8,134}\)
&
Write $T=R_8(x,0,y,z,0)\in\mathscr F_{8,134}, \quad x,y,z>0$.

Define
\(\displaystyle
P:=-\eta_6-3m_4=xy, \quad z=\frac{m_4}{P}, \quad S:=n-5-z=x+y.
\)

Then $x=P+zS+3S+3z+3-m_2, \quad y=S-x$. Thus \((x,y,z)\) is uniquely recovered.
\rule[-1.5mm]{0pt}{1.5mm}\\

\(\mathscr F_{9,1234}\)
&
Let \(a\) be the pendant multiplicity at the centre of
\(R_9=K_{1,3}\), and let \(b,c,d\) be those at the three leaves.

Put \quad \(s:=b+c+d,\quad p:=bc+bd+cd,\quad q:=bcd\).

Then
\(\displaystyle
q=\frac{-\eta_6-4m_4}{4},\quad a=\frac{m_4}{q}, \quad s=n-4-a, \quad p=m_2-(a+2)s.
\)

Thus \(\{b,c,d\}\) is the root multiset of $t^3-st^2+pt-q$, and the parameters are recovered up to permutation of the three
leaves.
\rule[-1.5mm]{0pt}{1.5mm}\\

\(\mathscr F_{10,1234}\)
&
Let \(a,b\) be the pendant multiplicities at the two internal vertices of \(R_{10}=P_4\), and \(c,d\) those at the corresponding endvertices.  

Put \quad
\(
u:=a+b,\quad v:=ab,\quad r:=c+d,\quad
w:=cd,\quad \ell:=ac+bd.
\)

Define
\(\displaystyle
\Pi:=-\eta_6-4m_4=uw,\qquad
\Omega:=\frac{\eta_5}{2}-m_3-\Pi=w.
\)

Then
\(\displaystyle
w=\Omega,\quad
u=\frac{\Pi}{\Omega},\quad
v=\frac{m_4}{\Omega},\quad
r=n-4-u, \quad
\ell=m_3-vr-\Pi-\Omega.
\)

Thus \(\{a,b\}\) and \(\{c,d\}\) are the root multisets of \(t^2-ut+v,\quad t^2-rt+w.\)

The known value \(\ell=ac+bd\) determines their pairing up to the reflection of the \(P_4\)-core.
\\

\end{longtable}
\endgroup

% ============================================================
% Table A.2: positive-family recovery for R_1--R_4
% ============================================================

\subsection{Positive families with cores
\texorpdfstring{\(R_1,\ldots,R_4\)}{R1,...,R4}}\label{app:positive-recovery-I}

Table~\ref{tab:positive-recovery-I} records the internal recovery certificates for the positive canonical families whose cores are \(R_1,\ldots,R_4\). Throughout this table,
\[
    E:=\eta_8,
    \qquad
    K:=-\frac{3\eta_7+6m_4+2\eta_8}{2}.
\]

\begingroup
\small
\setlength{\tabcolsep}{4pt}
\renewcommand{\arraystretch}{1.16}
\setlength{\LTleft}{0pt}
\setlength{\LTright}{0pt}
\setlength{\LTpre}{4pt}
\setlength{\LTpost}{4pt}

\begin{longtable}{@{}L{0.22\textwidth}L{0.74\textwidth}@{}}

\caption{Positive-family recovery certificates for the cores
\(R_1,R_2,R_3,R_4\).}
\label{tab:positive-recovery-I}\\

\toprule
Family or families & Recovery certificate\\
\midrule
\endfirsthead

\multicolumn{2}{c}{%
\tablename\ \thetable\ continued from the previous page%
}\\
\toprule
Family or families & Recovery certificate\\
\midrule
\endhead

\midrule
\multicolumn{2}{r}{Continued on the next page}\\
\endfoot

\bottomrule
\endlastfoot

% ============================================================
% R_1, support 1234
% ============================================================

\(\mathscr F_{1,1234}\)
&
Let \(a,b,c,d>0\) be the pendant multiplicities at the four leaves of \(R_1=K_{1,4}\), and put
\(
s:=a+b+c+d,\qquad
p:=ab+ac+ad+bc+bd+cd,
\)
\(
q:=abc+abd+acd+bcd,\qquad
r:=abcd.
\)

Then, with \(E:=\eta_8\),
\(\displaystyle
s=n-5,\qquad
r=\frac{E}{9},\qquad
q=m_4-r,\qquad
p=\frac{m_3-q}{2}.
\)
Thus \(\{a,b,c,d\}\) is the root multiset of $t^4-st^3+pt^2-qt+r$.

Every permutation of the four parameters is induced by an automorphism of \(R_1=K_{1,4}\).
\rule[-1.5mm]{0pt}{1.5mm}\\

% ============================================================
% R_2, supports 234 and 1234
% ============================================================

\(\mathscr F_{2,234}\),\newline
\(\mathscr F_{2,1234}\)
&
For the two families \(\mathscr F_{2,234}\) and \(\mathscr F_{2,1234}\), write $T=R_2(a,b,c,d)$, where \(a,b,c,d\) are the pendant multiplicities at \(v_1,v_2,v_3,v_4\), respectively, with \(a\ge0, b,c,d>0.\) Here \(a=0\) corresponds to support \(234\), whereas \(a>0\) corresponds to support \(1234\). The vertices \(v_3,v_4\) are interchanged by an automorphism of \(R_2\), so \(c,d\) are interchangeable. 

Put
\(\displaystyle
    u:=c+d,
    \qquad
    v:=cd,
    \qquad
    t:=a+b,
    \qquad
    S:=n-6.
\)

The data first recover
\(\displaystyle
    \delta:=a-b
      =
      \frac{\eta_3+4m_2-4S-6}{2},
\)
and
\(\displaystyle
    P:=ut
      =
      \frac{
        8S-4\delta+20-(\eta_5-6m_3+4m_2)
      }{2}.
\)

Thus the unordered pair \(\{t,u\}\) is the root multiset of $z^2-Sz+P$.

\\
&
The remaining coefficients determine the unique orientation $t=a+b, u=c+d$.

Once this orientation is known,
\(\displaystyle
    a=\frac{t+\delta}{2},
    b=\frac{t-\delta}{2}.
\)

Moreover, \(v=m_2-ab-ut-3a-4b-4u-5.\) Hence the interchangeable pair \(\{c,d\}\) is the root multiset of $z^2-uz+v$.

The recovered value of \(a\) distinguishes the two support orbits.
\rule[-1.5mm]{0pt}{1.5mm}\\

% ============================================================
% R_3, support 1234
% ============================================================

\(\mathscr F_{3,1234}\)
&
Let \(a>0\) be the pendant multiplicity at the central vertex \(v_1\) of \(R_3\), and let \(b,c,d>0\) be those at the three equivalent endvertices \(v_2,v_3,v_4\). 

Put
\(\displaystyle
    s:=b+c+d,\qquad
    p:=bc+bd+cd,\qquad
    r:=bcd.
\)

With
\(\displaystyle
    E:=\eta_8,\quad
    K:=-\frac{3\eta_7+6m_4+2\eta_8}{2},
\)
the central parameter is recovered from \(\displaystyle a=\frac{6E}{K-2E}.\)

Then \(s=n-a-7, p=m_2-as-3a-5s-9,\) and
\(\displaystyle
    r=\frac{E/a-2p-s}{3}.
\)

Hence \(\{b,c,d\}\) is the root multiset of $t^3-st^2+pt-r$.

Since every permutation of \(b,c,d\) is induced by an automorphism of \(R_3\), the pendant parameters are uniquely recovered up to a core automorphism.
\rule[-1.5mm]{0pt}{1.5mm}\\

% ============================================================
% R_4, supports 34, 134 and 1234
% ============================================================

\(\mathscr F_{4,34}\),\newline
\(\mathscr F_{4,134}\),\newline
\(\mathscr F_{4,1234}\)
&
For the three families \(\mathscr F_{4,34}\), \(\mathscr F_{4,134}\), and \(\mathscr F_{4,1234}\), write \(T=R_4(a,b,c,d),\) where \(a,b\ge0\) and \(c,d>0\) are the pendant multiplicities at the four relevant core vertices, ordered according to the chosen support representative. 

Put
\( x:=a+b, \qquad p:=ab, \qquad y:=c+d,\)
\(q:=cd, \qquad \ell:=ac+bd, \qquad S:=n-7.\)
       
Two zero entries in \(\{a,b\}\) give support \(34\), one zero entry gives support \(134\), and no zero entries give support \(1234\).
\\

&
The data first give
\(\displaystyle
    x
      =
      \frac{\eta_3+4m_2-2S-8}{2},
    \quad
    y=S-x.
\)
They also give \(A:=p+q=m_2-xy-4x-5y-10.\)

Define \( B:=m_3-3xy-4x-6y-4.\) Then \(B=p(y+2)+q(x+4).\)
\\

&
Suppose first that $y-x-2\neq0$. Since \(q=A-p\), the preceding equation gives
\(\displaystyle
    p
      =
      \frac{B-A(x+4)}{y-x-2},
    q=A-p.
\)
The mixed term is then recovered from
\(\displaystyle
    \ell
      =
      \frac{
        \eta_5
        -6py-8p
        -6xq-12xy
        -14x-20q-26y-16
      }{4}.
\)
\\

&
It remains to consider the exceptional case $y=x+2$.

In this case,
\(\displaystyle
    \ell
      =
      \frac{
        \eta_5
        -(6x+20)A
        -12xy-14x-26y-16
      }{4}.
\)

Put \(D:=E-\ell-(2x+1)A.\) Then \(p\) is an integral root of \(3z^2-(3A+4)z+D=0.\) This equation has at most one integral root, because the sum of its two roots is
\(\displaystyle
    A+\frac{4}{3}.
\)
Thus \(p\) is uniquely determined, and $q=A-p$.

\\

&
The unordered pairs \(\{a,b\}\) and \(\{c,d\}\) are now the root multisets of $z^2-xz+p$ and $z^2-yz+q$, respectively.

There are at most two ways to pair these two multisets. They give the mixed terms
\(
    ac+bd
    \quad\text{and}\quad
    ad+bc,
\)
whose difference is $(a-b)(c-d)$.

Hence the known value \(\ell\) determines the pairing unless one of the two pairs has equal entries. In that case, the two pairings are already related by the reflection of the path \(R_4=P_7\).
\\

\end{longtable}
\endgroup

% ============================================================
% Table A.3: positive-family recovery for R_5--R_8
% ============================================================

\subsection{Positive families with cores
\texorpdfstring{\(R_5,\ldots,R_8\)}{R5,...,R8}}
\label{app:positive-recovery-II}

Table~\ref{tab:positive-recovery-II} records the remaining internal recovery certificates for the positive canonical families. Throughout this table,
\[
    E:=\eta_8,
    \qquad
    K:=-\frac{3\eta_7+6m_4+2\eta_8}{2}.
\]

\begingroup
\small
\setlength{\tabcolsep}{4pt}
\renewcommand{\arraystretch}{1.16}
\setlength{\LTleft}{0pt}
\setlength{\LTright}{0pt}
\setlength{\LTpre}{4pt}
\setlength{\LTpost}{4pt}

\begin{longtable}{@{}L{0.22\textwidth}L{0.74\textwidth}@{}}

\caption{Positive-family recovery certificates for the cores
\(R_5,R_6,R_7,R_8\).}
\label{tab:positive-recovery-II}\\

\toprule
Family or families & Recovery certificate\\
\midrule
\endfirsthead

\multicolumn{2}{c}{%
\tablename\ \thetable\ continued from the previous page%
}\\
\toprule
Family or families & Recovery certificate\\
\midrule
\endhead

\midrule
\multicolumn{2}{r}{Continued on the next page}\\
\endfoot

\bottomrule
\endlastfoot

% ============================================================
% R_5, support 1234
% ============================================================

\(\mathscr F_{5,1234}\)
&
Let \(T=R_5(a,b,c,d)\in\mathscr F_{5,1234},\quad a,b,c,d>0,\) where \(c,d\) are interchangeable under a core automorphism.

Put $R:=E/4=abcd$. The distinguished parameter \(a\) is recovered from
\(\displaystyle
    a=\frac{3E}{2K-E}.
\)

Next set \(S:=n-5-a=b+c+d,\quad U:=c+d,\quad V:=cd.\)

From \(m_4\),
\(\displaystyle
    B:=bU
      =\frac{m_4-R-R/a}{a}.
\)
Then \(V=m_2-aS-B-2a-3S-2.\) Since $R=abV$, we recover
\(\displaystyle
    b=\frac{R}{aV},
    U=S-b.
\)
Hence the interchangeable pair \(\{c,d\}\) is the root multiset of $z^2-Uz+V$.

Thus \(a\) and \(b\) are uniquely recovered, while \(\{c,d\}\) is recovered up to the corresponding core automorphism.
\rule[-1.5mm]{0pt}{1.5mm}\\

% ============================================================
% R_5, support 2345
% ============================================================

\(\mathscr F_{5,2345}\)
&
Let \(a,b,c,d>0\) be the pendant parameters, with \(b,c\) interchangeable under a core automorphism. The distinguished terminal parameter is recovered from
\(\displaystyle
    d=\frac{6E}{K-2E}.
\)

Put \(S:=n-5-d=a+b+c,\quad U:=b+c,\quad V:=bc.\)

Since
\(\displaystyle
    m_4=E+\frac{E}{d}+dV,
\)
the product \(V\) is recovered from
\(\displaystyle
    V=\frac{m_4-E-E/d}{d}.
\)
Moreover, $E=aVd$, so
\(\displaystyle
    a=\frac{E}{Vd},
    U=S-a.
\)
Hence the interchangeable pair \(\{b,c\}\) is the root multiset of \(z^2-Uz+V.\) Thus \(a\) and \(d\) are uniquely recovered, while \(\{b,c\}\) is recovered up to the corresponding core automorphism.
\rule[-1.5mm]{0pt}{1.5mm}\\

% ============================================================
% R_6, support 1234
% ============================================================

\(\mathscr F_{6,1234}\)
&
Let \(a,b,c,d>0\) be the pendant parameters, with \(c,d\) interchangeable under a core automorphism. 

Put $u:=c+d, v:=cd$. The distinguished parameter \(a\) is recovered from 
\(\displaystyle
a=\frac{6E}{K-2E}.
\)

Next set \(S:=n-6-a=b+u=b+c+d.\)

Define
\(\displaystyle
    B:=\frac{2m_4-(a+2)E/a}{a}. \quad \text{Then} \quad B=b(u+2).
\)

Now put \(C:=2\bigl(m_2-aS-2a-B-4S-5\bigr)+S.\) In terms of the pendant parameters, \( C=2v+u-3b=E/ab-3b.\) Consequently, \(b\) is a positive root of
\(3z^2+Cz-E/a=0.\)
Since the product of the two roots is $-E/3a<0$, this quadratic has exactly one positive root. Hence \(b\) is uniquely recovered.

Finally,
\(\displaystyle
    u=S-b,
    v=\frac{E/(ab)-u}{2}.
\)
Thus the interchangeable pair \(\{c,d\}\) is the root multiset of $z^2-uz+v$.

Hence \(a\) and \(b\) are uniquely recovered, while \(\{c,d\}\) is recovered up to the corresponding core automorphism.
\rule[-1.5mm]{0pt}{1.5mm}\\

% ============================================================
% R_7, supports 134 and 1234
% ============================================================

\(\mathscr F_{7,134}\),\newline
\(\mathscr F_{7,1234}\)
&
For the two families
\(\mathscr F_{7,134}\) and \(\mathscr F_{7,1234}\), allow $a,b\ge0,\quad a+b>0,\quad c,d>0$, where the parameters are attached to \(v_1,v_2,v_3,v_4\),
respectively.  The reflection of \(R_7=P_6\) acts by \((a,b,c,d)\longmapsto(b,a,d,c).\)

Put
\(
    x:=a+b,\qquad p:=ab,\qquad
    y:=c+d,\qquad q:=cd,
\)
\(
    \ell:=ac+bd,\qquad S:=n-6=x+y.
\)

For the support orbit represented by \(134\), exactly one of \(a,b\) is zero, whereas \(a,b>0\) for support \(1234\).

The data first recover
\(\displaystyle
    R:=\frac{\eta_3+4m_2-2S-6}{2}=p+x.
\)
The value \(x=a+b\) is the unique admissible candidate determined by the candidate system in Appendix~\ref{app:R7-candidate}.  

Once \(x\) is known,
\(
    y=S-x,\quad
    p=R-x,\quad
    q=m_2-p-xy-3x-4y-6,
\)
and \(\ell=pq+py+p+xq+xy+q-m_4.\)

Hence \(\{a,b\}\) and \(\{c,d\}\) are the root multisets of $z^2-xz+p$, and $z^2-yz+q$, respectively. Their two possible pairings give $ac+bd$ and $ad+bc$,
whose difference is $(a-b)(c-d)$.

Thus the known value \(\ell=ac+bd\) determines the pairing unless one of the two pairs has equal entries; in that case the two pairings are related by the reflection of \(R_7=P_6\).

Finally, the recovered pair \(\{a,b\}\) determines the support orbit: one zero entry gives \(134\), while no zero entry gives \(1234\).
\rule[-1.5mm]{0pt}{1.5mm}\\

% ============================================================
% R_7, support 1346
% ============================================================

\(\mathscr F_{7,1346}\)
&
Let \(a,b,c,d>0\) be the pendant parameters attached to
\(v_1,v_3,v_4,v_6\), respectively, and put \(S:=n-6=a+b+c+d.\)

Define
\(\displaystyle
    R:=\frac{\eta_3+4m_2}{2}-S-3=a+cd,
\)
\(\displaystyle
    X:=\frac{-\eta_7-4m_4}{2}=abc(d+1),
    \qquad
    Y:=E-X=acd(b+1),
\)
\(\displaystyle
    Z:=-\frac{3m_3-\eta_4-m_2+3S+R+7}{2}
      =a(b+d),
\)
\(\displaystyle
    J:=2m_2-m_3-3m_4-\eta_6-5S-2Z-Y-11
      =ab(d+1).
\)

Since \(J>0\), \(\displaystyle c=\frac{X}{J}\) is uniquely recovered.

Next put $U:=S-c=a+b+d$. Since $Z=a(b+d)$, the two numbers \(a\) and \(b+d\) form the root multiset of $z^2-Uz+Z$.

For a candidate root \(z\) for \(a\), define
\(\displaystyle
    h_z:=U-z,\quad
    d_z:=\frac{R-z}{c},\quad
    b_z:=h_z-d_z,
\)
and
\(
    J_z:=z b_z(d_z+1),\quad
    Y_z:=zc\,d_z(b_z+1).
\)
The correct root is the unique candidate satisfying $J_z=J,\quad Y_z=Y$.

Once it is selected,
\(\displaystyle
    a=z,\quad
    d=\frac{R-a}{c},\quad
    b=U-a-d.
\)
Thus all four pendant parameters are uniquely recovered in the chosen support representative \(1346\).
\rule[-1.5mm]{0pt}{1.5mm}\\

% ============================================================
% R_8, support 1234
% ============================================================

\(\mathscr F_{8,1234}\)
&
Let \(a,b,c,d>0\) be the pendant parameters. The distinguished parameter \(b\) is recovered from
\(\displaystyle
    b=\frac{3E}{K-2E}.
\)

Put \(P:=ac.\) Since \(m_4=E+E/b+bP,\) we obtain
\(\displaystyle
    P=\frac{m_4-E-E/b}{b}.
\)
Moreover, since $E=Pbd$, the parameter \(d\) is recovered from $d=E/Pb$.

Now set \(S:=n-5-b=a+c+d,\quad X:=S-d=a+c.\)
The identity
\(\displaystyle
    m_2=(b+d+3)X+P+bd+2b+3d+3-a
\)
then gives $a=(b+d+3)X+P+bd+2b+3d+3-m_2$, and finally \(c=X-a.\) 

Thus all four pendant parameters \(a,b,c,d\) are uniquely recovered.
\rule[-1.5mm]{0pt}{1.5mm}\\

% ============================================================
% R_8, support 1345
% ============================================================

\(\mathscr F_{8,1345}\)
&
Let \(a,b,c,d>0\) be the pendant parameters attached to \(v_1,v_3,v_4,v_5\), respectively. The reflection of the path \(R_8\) simultaneously interchanges
\(\displaystyle
    (a,b)\longleftrightarrow(d,c).
\)

Put \(u:=a+d,\quad v:=ad,\quad r:=b+c,\quad w:=bc,\) and \(\ell:=ab+cd.\)

Define \(P:=m_4-E=uw,\quad A:=-\eta_6-4E=3P+vr+\ell,\) and
\(
    S:=n-5=u+r.
\)

Further, put \(\displaystyle C:=m_3-A+2P-m_2+2S+3=w-v =w-\frac{E}{w}.\)

Hence \(w\) is the unique positive root of $z^2-Cz-E=0$, since the product of the two roots is \(-E<0\).

Once \(w\) is known, recover
\(\displaystyle
    u=\frac{P}{w},\quad
    v=\frac{E}{w},\quad
    r=S-u,\quad
    \ell=A-3P-vr.
\)

Thus the unordered pairs \(\{a,d\}\) and \(\{b,c\}\) are the root multisets of

\(
    z^2-uz+v
    \quad\text{and}\quad
    z^2-rz+w,
\)
respectively.

There are at most two ways to pair these two unordered pairs. They give the mixed terms
\(
    ab+cd
    \quad\text{and}\quad
    ac+bd,
\)
whose difference is $(a-d)(b-c)$.

Therefore the known value \(\ell=ab+cd\) determines the pairing unless one of the two pairs has equal entries.  In that exceptional case the two pairings are already related by the reflection of \(R_8\). Hence the pendant parameters are recovered up to the core reflection.

\\

\end{longtable}
\endgroup

\subsection{Candidate system for the combined \(R_7\) row}\label{app:R7-candidate}
This subsection supplies the candidate test used in the recovery certificate for \(\mathscr F_{7,134}\) and \(\mathscr F_{7,1234}\).
Recall that
\(
    x=a+b, p=ab, y=c+d, q=cd,
\)
and that the data determine \(S=x+y,\quad R=p+x.\)

For a candidate \(z\) for \(x\), put
\(
    y_z:=S-z,\quad
    p_z:=R-z,\quad
    q_z:=m_2-p_z-zy_z-3z-4y_z-6,
\)
and define
\begin{align}
    T_z&:=p_zy_z+2p_z+zq_z+2zy_z+2z+3q_z+3y_z+1,
    \notag\\
    U_z&:=2p_zq_z+p_zy_z+zq_z.
    \label{eq:R7-candidate-system}
\end{align}
The true value \(z=x\) satisfies
\(
    T_z=m_3, U_z=E.
\)
Moreover,
\begin{equation}\label{eq:R7-candidate-quadratic}
    T_z+U_z
      =
      (2R+3)z^2
      +(2R+3)(1-S)z
      +C_0,
\end{equation}
where $C_0=(2R+3)m_2-2R^2-6RS-13R-9S-17$.

% =========================================================
% Bibliography
% =========================================================

\end{document}